\documentclass[12pt,a4paper]{amsart}

\usepackage{euscript,amsfonts,amssymb,amsmath,amscd,amsthm,enumerate,hyperref,mathrsfs}

\usepackage{tikz,bm,float}
\usetikzlibrary{calc}
\colorlet{lgray}{white!85!black}
\colorlet{lred}{white!75!red}

\usepackage{comment}

\usepackage{graphicx}

\usepackage{color}

\usepackage[margin=1.1in]{geometry}
\newtheorem{theorem}{Theorem} 
\newtheorem*{theorem*}{Theorem}

\newtheorem{proposition}[theorem]{Proposition}

\newtheorem{corollary}[theorem]{Corollary}

\theoremstyle{remark}
\newtheorem{remark}[theorem]{Remark}
\newtheorem{example}[theorem]{Example}

\numberwithin{equation}{section} \numberwithin{theorem}{section}

\newcommand{\la}{\lambda}

\newcommand{\Z}{\mathbb Z}

\newcommand{\R}{\mathbb R}

\newcommand{\C}{\mathbb C}

\usepackage{MnSymbol}

\usepackage{skak}

\usepackage{adjustbox}
\usetikzlibrary{arrows}
\usetikzlibrary{decorations.pathreplacing}

\title[Color-position symmetry]
{Color-position symmetry in interacting particle systems}

\author{Alexei Borodin}

\address[Alexei Borodin]{Department of Mathematics, MIT, Cambridge, USA, and Institute for Information
Transmission Problems, Moscow, Russia. E-mail: borodin@math.mit.edu}

\author{Alexey Bufetov}

\address[Alexey Bufetov]{Hausdorff Center for Mathematics \& Institute for Applied Mathematics, University of Bonn, Germany. E-mail: alexey.bufetov@gmail.com}

\begin{document}

\begin{abstract}
	
We prove a color-position symmetry for a class of ASEP-like interacting particle systems with discrete time on the one-dimensional lattice. The full space-time inhomogeneity of our systems allows to apply the result to colored (or multi-species) ASEP and stochastic vertex models for a certain class of initial/boundary conditions, generalizing previous results of Amir-Angel-Valko and Borodin-Wheeler. We are also able to use the symmetry, together with previously known results for uncolored models, to find novel asymptotic behavior of the second class particles in several situations.

\end{abstract}

\maketitle

\section{Introduction}

The first hint to color-position symmetry in interacting particle systems
dates back to an elegant work of Ferrari-Kipnis \cite{FK}, who showed that the
behavior of a second class particle of the Totally Asymmetric Simple Exclusion
Process (or TASEP) in a rarefaction fan can be obtained from the density
function and characteristics of the first class particles. Let us start by
giving an exact formulation of one of their results.

The TASEP is a prototypical interacting particle system on $\mathbb{Z}$ that
consists of particles occupying integer sites, no more than one particle per
site. It evolves in continuous time with particles jumping to the right by one,
provided that the target site is empty, using independent exponential clocks.
One distinguishes the homogeneous TASEP, when the rates of all clocks are equal
to 1, and the inhomogeneous one, when the clock rates may depend on space and
time. Since early 1970's, the TASEP has been extensively studied in hundreds of
papers, cf., e.g., the books by Liggett
\cite{L3}, Kipnis-Landim \cite{KL}, and references therein.

The \emph{colored}, or {multi-species} TASEP is a similar particle systems where particles have an additional feature called color, which is typically integer-valued. The color affects the evolution in the following way: Particles with lower color treat particles with higher color as holes. More exactly, whenever a ringing clock tells a particle to jump, it swaps positions with the site immediately to its right if and only if that site is either unoccupied or occupied by a particle of a higher color. For uniformity of presentation, it is convenient to treat unoccupied sites as being occupied by particles of color $+\infty$. \footnote{The convention of using particles with lower color as ``more important'' ones is not canonical, e.~g., \cite{BorWh} uses the opposite one. Our convention in this paper is aligned with the common usage of the terms "first class'' and ``second class'' particles, with the intention that the word ``class'' could be replaced by ``color'' without changing the evolution rules.}

Consider now the TASEP with the following, often called \emph{step} initial condition: First class particles (equivalently, particles of color 1) occupy all the negative integers. Also, let us place a single second class particle at site 0. Then \cite{FK} showed, by a simple coupling argument, that at any time $t>0$, the probability to find the second class particle at a position $>x$ is equal to the probability to find $x$ occupied by a first class particle. Together with the well-known large time behavior of the density of the first class particles (that goes back to \cite{R} for the step initial condition),
this implied, in particular, that at large times the second class particle is asymptotically uniformly distributed on $[-t,t]$.\footnote{\cite{FK} also proved more refined results on the asymptotic behavior of the second class particle at different times (it follows a characteristic). We omit those as in the present text we focus only on a fixed time picture.}

Amir-Angel-Valko \cite{AmAnV}, cf. also the previous paper of
Angel-Holroyd-Romik \cite{AnHR}, found a way to substantially generalize the
results of \cite{FK}.

First, they considered the lattice $\mathbb{Z}$ filled with particles of colors ranging also over $\mathbb{Z}$, one particle per each color. This puts positions and colors on the same footing and allows to interpret a particle configuration as a bijection $\pi_{p\to c}:\mathbb{Z}\to\mathbb{Z}$ such that $\pi_{p\to c}(x)$ is the color of the particle at position $x$. Alternatively, one can define a bijection $\pi_{c\to p}:\mathbb{Z}\to\mathbb{Z}$ by saying that $\pi_{c\to p}(c)$ is the position of the particle of color $c$. Clearly, $\pi_{p\to c}$ and $\pi_{c\to p}$ are mutually inverse.

Second, they considered a partially asymmetric version of the TASEP, known as PASEP or simply ASEP. Under the ASEP evolution, the particles can jump both right and left by 1 with different rates. Let us say that more formally.

Fix a parameter $q\ge 0$, and for any two neighboring sites $z,z+1\in\mathbb{Z}$ and $x\in [0,1]$, define a random \emph{asymmetric swap} $W_{(z,z+1),x}$ that acts on particles located at $z,z+1$ of a colored particle configuration in $\mathbb{Z}$ as depicted on Figure \ref{Fig1-intro}, see also the caption (nothing happens if the color of particles at $z,z+1$ is the same; recall also that unoccupied sites are treated as those occupied by particles of color $+\infty$).

\begin{figure}
	\centering
	\includegraphics[width=10cm]{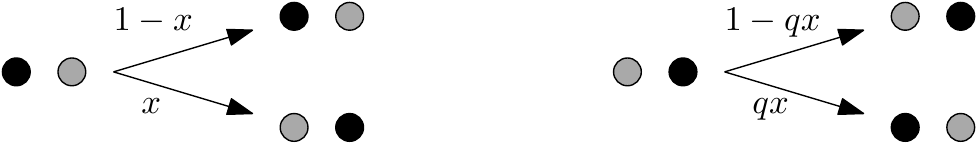}
	\caption{The action of a random asymmetric swap; the color of the black particle is assumed to be lower than that of the grey one.}  \label{Fig1-intro}
\end{figure}

The colored ASEP evolution is defined via applications of such asymmetric swaps
at random times that form independent Poisson processes of rate 1 for each pair
$\{z,z+1\}\subset\mathbb{Z}$. If the parameter $x$ is independent of the
position $z$ and the time $t$ of the swap, then one speaks of a homogeneous
ASEP, otherwise it is (spatially, temporarily, or doubly) inhomogeneous. The
TASEP corresponds to $q=0$ (note that we keep the asymmetry
parameter $q\in [0,1]$ constant).

Consider now a \emph{packed} initial configuration -- every $x\in\mathbb{Z}$ is
occupied by a particle of color $x$. Clearly, this corresponds to the two
(mutually inverse) permutations $\pi_{p\to c}$ and $\pi_{c\to p}$ being equal to
the identity. Amir-Angel-Valko \cite[Theorem 1.4]{AmAnV} proved the following
striking symmetry for the (homogeneous) colored ASEP: With the packed initial
condition, and for any $t\ge 0$, the two permutations  $\pi_{p\to c}$ and
$\pi_{c\to p}$ that parametrize the particle configuration at time $t$ have the
same distribution.

This symmetry readily implies the above stated second class particle result of \cite{FK}. Indeed, it suffices to notice that the colored ASEP can always be projected onto one with fewer colors by declaring all particles with colors from any interval $I\subset\Z$ as having the same color (and naturally extending the color ordering to this new color). Hence, uniting the colors from $(-\infty,-1)$ and calling the corresponding particles ``first class'', calling the particle of color 0 ``second class'', and uniting the colors from $(1,+\infty)$ and calling them holes, we obtain the process from \cite{FK}; the fact that $\mathrm{Prob}\{\pi_{c\to p}(0)>x\}=\mathrm{Prob}\{\pi_{p\to c}(0)>x\}$ delivers the result\footnote{via interpreting the right-hand side as the probability of having 0 unoccupied at time $t$ in the conventional ASEP started with $(-\infty,x)$ filled by particles, and applying the particle-hole involution.}.

Here is the main symmetry result of the present paper.

\smallskip

\noindent\textbf{Symmetry Theorem} (cf.~Theorem \ref{th1} below).
\emph{Consider the packed initial configuration, any $x\in\Z$ is occupied by a
particle of color $x$, and apply an arbitrary finite sequence
$W_{(z_1,z_1+1),x_1}, \dots, W_{(z_k,z_k+1),x_k}$ of random asymmetric swaps to
it (the asymmetry parameter $q\in [0,1]$ remains fixed, but $z_i$'s in
$\mathbb{Z}$ and $x_i$'s in $[0,1]$ are arbitrary). Then the distribution of
the position-to-color permutation $\pi_{p\to c}$ for the resulting configuration is
equal to the distribution of the color-to-position permutation $\pi_{c\to p}$
that corresponds to the configuration obtained by applying the same set of
asymmetric swaps to the packed initial configuration \emph{in the opposite
order}.}

\smallskip

We find this result rather surprising, and at the moment we cannot offer a
conceptual understanding of why it should hold. The proof that we give is an
induction argument that is much more a verification than a derivation.

For $x_k\equiv 1$ this theorem  is identical to \cite[Lemma 3.1]{AmAnV}.
However, the introduction of (both spatial and temporal) inhomogeneities is
important; for example, it allows to access the stochastic six vertex model,
reproving and extending a certain symmetry of that model observed in
\cite{BorWh}.

We see the main probabilistic contribution of the present work in
providing new ways of utilizing the Symmetry Theorem for large time analysis of
ASEP-like particle systems. As we demonstrate below, it offers a number of
delightful and nontrivial corollaries. We also believe that what we do in this work
does not exhaust a list of such by a good margin.

\subsection*{Second class particle in the inhomogeneous ASEP} Consider the
doubly inhomogeneous colored ASEP with bond rates $x(z,t)$ depending both on
space $z$ and time $t$ (and satisfying natural technical assumptions), and start
it with the packed initial condition. Let $\pi_{p\to c}$ be the
position-to-color permutation of the resulting configuration at time $T>0$. Then, by Symmetry Theorem, 
it has the same distribution as $\pi_{c\to p}$ for the configuration at time $T$
of the \emph{time-reversed} colored ASEP with bond rates $\hat x(z,t)=x(z,T-t)$
and also started from the packed initial condition.

When applied to the second class particle problem, this gives the following result. Consider the ASEP with first-class particles filling $(-\infty,-1)$ and a single class particle at 0 initially, with doubly inhomogeneous bond rates $x(z,t)$. For the time-reversed process, take the single-colored ASEP with bond rates $\hat x(z,t)=x(z,T-t)$ and the $y$\emph{-shifted} step initial condition: $(-\infty,y]$ is filled by particles for some $y\in\mathbb{Z}$. Then, cf. Theorem \ref{th:inh2cl} below, we show that
\begin{gather*}
\mathrm{Prob}\{\text{the second class particle is weakly to the left of position $y$ at time $T$} \}\\
=\mathrm{Prob}\{ \text{position $y$ is occupied at time $T$ in time-reversed ASEP with $y$-shifted step IC} \}.
\end{gather*}

Whenever asymptotic results for the density of the conventional inhomogeneous ASEP are available, e.g., for smoothly changing spatial inhomogeneities as proved by Bahadoran \cite{B}, they would thus immediately give the asymptotic distribution of our second class particle.

Observe that for different values of $y$, the time-reversed ASEPs are \emph{different}, as they involve different $y$-shifts of the step initial condition that are not equivalent if the bond rates are spatially inhomogeneous. Also, in general, there seems to be no easy way to relate these time-reversed ASEPs to the original ASEP, unlike the homogeneous case.

\subsection*{Finite perturbation of the step initial condition} Fix an integer
$L\ge 0$, and consider the (homogeneous) ASEP on $\mathbb{Z}$ with first and
second class particles and with the following initial condition: First class
particles occupy all positions in $(-\infty,-L)$, there are no particles in
$(L,+\infty)$, and on the segment $[-L,L]$ we have an arbitrary mix of the
first and second class particles and holes.

Let $S_1(t)\le S_2(t)\le \dots\le S_N(t)$ be the positions of all the second class particles at time $t\ge 0$. Theorem \ref{th:finPert} below completely described the limiting distribution function as $t\to\infty$ for each $S_k$, $1\le k\le N$ (but not their joint distribution). Rather than giving the full statement here, let us provide an example. Consider the initial condition depicted in Figure \ref{Fig2-intro}; here $L=1$, the only second class particle is at $-1$, 0 is occupied by a first class particle, and 1 is unoccupied.

\begin{figure}
	\centering
	\includegraphics[width=9cm]{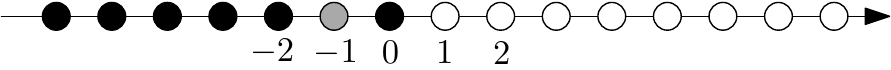}
\caption{An example of a finitely perturbed step initial condition.}
\label{Fig2-intro}
\end{figure}

Then Theorem \ref{th:finPert} yields the following limit for the asymptotic position of the second class particle, cf. Example \ref{ex:quadratic} below:
\begin{equation*}
\lim_{t \to \infty} \mathrm{Prob} \left\{\frac{S_1(t)}{t}\le y \right\} =\begin{cases} 0, &y\le q-1,\\\alpha(y)+(1-q)\alpha(y)(1-\alpha(y)),&q-1<y<1-q,\\1,&y\ge 1-q, \end{cases}
\end{equation*}
where $\alpha(y)=\frac12 (1+\frac{y}{1-q})$. To our knowledge, already for this simple perturbation of the step initial condition, the limit theorem is new.

Note that the above limiting distribution substantially depends on $q$, thus
disagreeing with the usual intuition that limiting behavior of ASEPs with
different $q$ is expected to be the same up to an appropriate time change.

In the TASEP ($q=0$) and one second class particle ($N=1$) case, similar results can be deduced from the main result of Cator-Pimentel
\cite{CP}, but their methods do not seem to be extendable to the $q\ne 0$ situation.

\subsection*{Dissolving GUE-GUE shock for the TASEP} For this application we focus on the (homogeneous) TASEP (i.~e., $q=0$) with the initial condition depicted in Figure \ref{Fig3-intro}; the parameter $L$ and the time $t$ of the evolution of the system are assumed to be dependent and growing.

When $t\sim L$, the first particles of the left group start reaching 0, while the second class particle at zero starts moving (it cannot move before because of the first class particles in front of it). The speed of the particles from the left group far exceeds that of the last particles of the right group, which creates a macroscopic shock at 0. We positioned the second class particle so that it would initially be at the location of that shock. In certain situations, one can use the evolution of a second class particle placed at the shock location to define the shock at the microscopic scale, cf. Ferrari-Kipnis-Saada \cite{FKS}, Ferrari \cite{F2}, Derrida-Lebowitz-Speer \cite{DLS}.

With our initial condition, as $t$ becomes much larger than $L$, the
macroscopic shock disappears (the asymptotic density profile becomes
indistinguishable from $L=0$). It remains visible, however, at the level of
$t^{1/3}$ fluctuations as a ``GUE-GUE shock'', cf. Ferrari-Nejjar \cite{FN1}, \cite{FN2}, Quastel-Rahman \cite{QR}, until $t$ becomes
much larger than $L^{3/2}$. We investigate the asymptotic behavior of the second
class particle at times $t\sim \mathrm{const}\cdot L^{3/2}$, when a diminishing
shock effect is still present at the fluctuation level.

We prove, see Theorem \ref{th:shock2cl} below, that if $L = \lfloor c t^{2/3} \rfloor$,
$c \in \R_{>0}$, $y \in \R$, and $S_1(t)$ is the position of the second class
particle, then
\begin{equation*}
\lim_{t \to \infty} \mathrm{Prob} \left( \frac{S_1(t)}{t^{2/3}} \le y \right) = \mathrm{Prob} \left( \mathcal{A}_2 \left( \frac{-y+c}{2^{1/3}} \right) - \mathcal{A}_2 \left( \frac{-y-c}{2^{1/3}} \right) \ge 2^{4/3} yc \right),
\end{equation*}
where $\mathcal{A}_2$ is the Airy$_2$ process.

It is also interesting to investigate the behavior of the second class particle
for other dependencies of the form $L\sim t^{a}$ as $t\to\infty$; we hope to
address that in a later work.

\begin{figure}
	\centering
	\includegraphics[width=12cm]{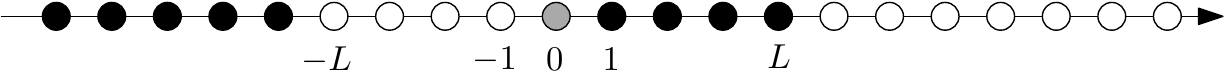}
	\caption{Shock inducing initial condition; the 2nd class particle is in grey.}
	\label{Fig3-intro}
\end{figure}

\subsection*{The symmetry for the stochastic six vertex model}\ One way to think of the stochastic six vertex model is as
of a discrete time analog of the ASEP. It was originally introduced and studied by Bethe ansatz methods by Gwa-Spohn in 1992 \cite{GS}. Much more recently, it was realized that this model is much more powerful than the ASEP due to its deep connection to the extensive algebraic structure of quantum groups; one corollary of that is that one can obtain many other probabilistic systems via its \emph{fusion} and subsequent degeneration. A lecture style exposition of the fusion procedure can be found in \cite{BP3}. There was also a flurry of recent activity in proving different asymptotic results for it, see, e.g., Borodin-Corwin-Gorin \cite{BCG}, Aggarwal \cite{Ag1, Ag2}, Aggarwal-Borodin \cite{AB}, Reshetikhin-Sridhar \cite{RS}, Corwin-Ghosal-Shen-Tsai \cite{CGST}, Borodin \cite{Bo}, Borodin-Olshanski \cite{BO}, Borodin-Gorin \cite{BG}, Shen-Tsai \cite{ST}.

The colored (or multi-species) six vertex model is even more recent; it first appeared in the work of Kuniba-Mangazeev-Maruyama-Okado \cite{KMMO}, see also Kuan \cite{K}. The only asymptotic results currently available in the colored case originate from a certain distributional match between colored and uncolored models discovered by Borodin-Wheeler \cite{BorWh} together with previously known results in the uncolored case, see Section 1.8 of \cite{BorWh} for more details. That match was the original motivation of the present work.

We show, in Corollary \ref{th:symmetryV} below, that our Symmetry Theorem directly applies to the colored six vertex model on domains with arbitrary monotone lattice boundaries, where it produces an identity between partition functions for a domain and its image rotated by 180 degrees, with appropriate boundary conditions. One consequence of this result is an expression, in Theorem \ref{th:HFmultiCvert} below, of the random vector representing the colored height function of the colored model at a vertex in terms of a Hall-Littlewood process on partitions. (A similar result appeared in \cite{BorWh} for rectangular domains, yet that result is still different from the one coming from Theorem \ref{th:HFmultiCvert}, with neither of the two implying the other.)
 We anticipate that this will lead to significant asymptotic results, as Hall-Littlewood processes (and Schur processes that arise when $q=0$) are known to be asymptotically accessible, cf. Dimitrov \cite{D}, Borodin \cite{Bo},
Corwin-Dimitrov \cite{CD}, and references therein. However, we decided not to address such asymptotics in the present text.

\smallskip

The organization of the paper is similar to that of the introduction -- we start with a proof of the Symmetry Theorem (in the language of permutations) and then proceed to the above described applications following the same order.

\subsection*{Acknowledgments} We are grateful to A.~Aggarwal and P. L.~Ferrari for very valuable comments.
The work of A.~Borodin was partially supported by the NSF grants DMS-1607901 and DMS-1664619.

\section{Color-position symmetry: permutations}
\label{subsec:perm1}

In this section we formulate and prove our main combinatorial result.

\subsection{Preliminaries}

Let $S_N$ be the permutation group on $N$ elements, and let $\C [S_N]$ be its group algebra. We use the notational convention $\sigma_2 \sigma_1 (i) = \sigma_2 (\sigma_1 (i))$ for $\sigma_1, \sigma_2 \in S_N$. Let $t = (A, A+1)$ be a transposition of two neighboring integers, and let $q \in \C$ be a fixed parameter. For $x \in \C$ we define a linear operator $w_{t,x}: \C [S_n] \to \C [S_n]$ via the following action on $\sigma \in S_N$:
\begin{equation*}
w_{t,x} ( \sigma) = \begin{cases}
(1 - x) \sigma + x t \sigma, \qquad & \mbox{if $\sigma^{-1} (A)<\sigma^{-1} (A+1)$}, \\
(1 - q x) \sigma + q x t \sigma, \qquad & \mbox{if $\sigma^{-1} (A) > \sigma^{-1} (A+1)$}.
\end{cases}
\end{equation*}

It is helpful to think about this and subsequent notions in terms of certain colored pictures. Consider $N$ vertical segments numbered from $1$ to $N$, and assign to each of this segment a color. Colors are also numbered from $1$ to $N$, and we use each color exactly once. Then a permutation $\sigma$ encodes the information that vertical line $\sigma(i)$ has color $i$, $1 \le i \le N$ (see Figure \ref{fig:1}). One can also say that vertical line $j$ is colored by $\sigma^{-1} (j)$. The linear operator $w_{t,x}$ can be depicted by a cross which involves the interaction between two vertical lines, see Figure \ref{fig:2}. This cross has a weight depending on colors involved.
The expansion of $w_{t,x} ( \sigma)$ can be encoded as in Figure \ref{fig:3}, left panel: The cross must have one of two allowed types, the top coloring is given by either $\sigma$ or $t \sigma$, and such picture has weight of cross.

\begin{figure}
\begin{center}
\noindent{\scalebox{0.27}{\includegraphics[height=7cm]{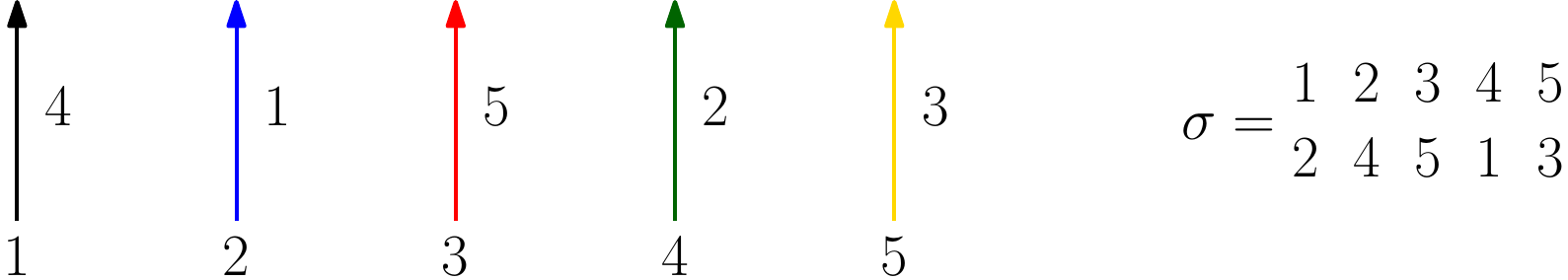}}}
\caption{A pictorial representation of a permutation, $N=5$. \label{fig:1}}
\end{center}
\end{figure}

\begin{figure}
\begin{center}
\noindent{\scalebox{0.7}{\includegraphics[height=7cm]{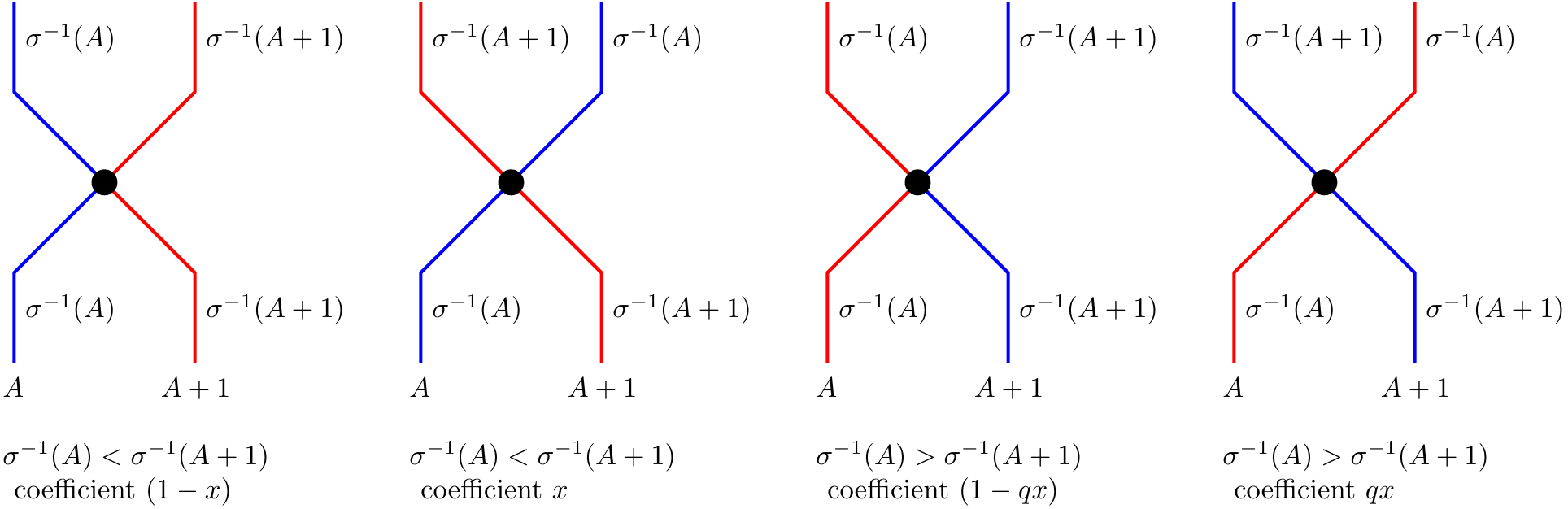}}}
\caption{A pictorial representation of $w_{t,x}$. \label{fig:2}}
\end{center}
\end{figure}

For transpositions $t_1, \dots, t_n$ and parameters $x_1, \dots, x_n$ one has an expansion
\begin{equation}
\label{eq:model-def}
w_{t_n, x_n} w_{t_{n-1}, x_{n-1}} \dots w_{t_1, x_1} s =: \sum_{\pi \in S_N} f_n (s \to \pi) \pi, \qquad s \in S_N,
\end{equation}
which defines the coefficients $\{ f_n (s \to \pi) \}_{s, \pi \in S_N}$ (we omit the dependence on $\{ t_i \}$ and $\{ x_i \}$ in the notation for them). In pictorial terms, we consider vertical lines with $n$ crosses. To each possible configuration of colors we assign a weight which is equal to the product of weights of all crosses. Then, to obtain the coefficients $f_n (s \to \pi)$, we set the bottom coloring to be equal to $s$, the top coloring to be equal to $\pi$, and we sum the weights of all possible configurations with such bottom and top (see Figure \ref{fig:3} for an example).

\begin{figure}
\begin{center}
\noindent{\scalebox{0.55}{\includegraphics[height=7cm]{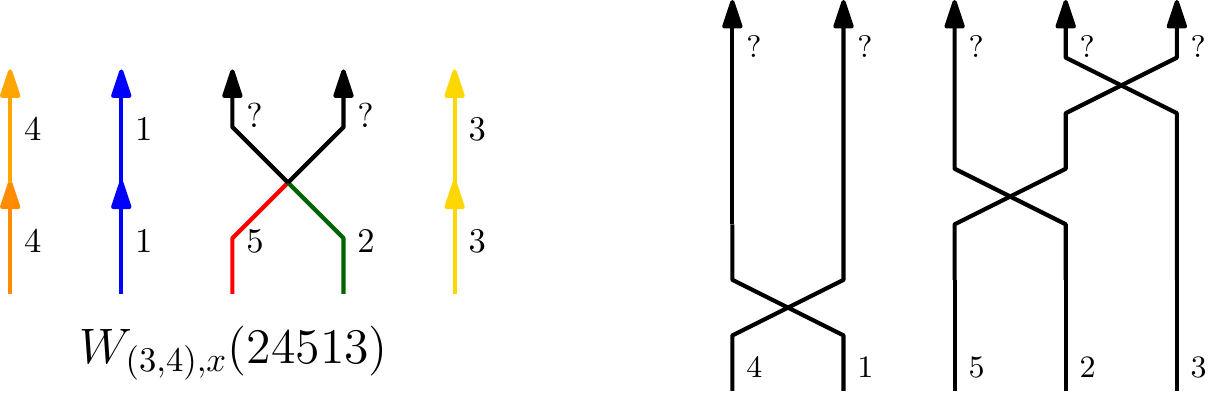}}}
\caption{Left panel: A configuration with one cross. There are two possible permutations on the top. Right panel: A configuration with $n=3$ crosses. All possible colorings with a fixed coloring $s$ at the bottom correspond to the sum $\sum_{\pi \in S_N} f_3 (s \to \pi) \pi$, where the coloring of the top row is $\pi$. \label{fig:3}}
\end{center}
\end{figure}

\begin{proposition}
\label{prop:main}
Let $T = (i,i+1)$ be a transposition, and $s$, $\pi$ be arbitrary permutations. Let $t_1, \dots, t_n$ and $x_1, \dots, x_n$ be an arbitrary choice of parameters as above. Assume that $s (i) < s (i+1)$.

Then we have
\begin{equation}
\label{eq:partAprop}
f_n \left( s T \to \pi \right) = \begin{cases}
f_n ( s \to \pi T) + (1-q) f_n (s \to \pi), \qquad & \mbox{\textnormal{if} $\pi (i) >\pi (i+1)$}, \\
q f_n ( s \to \pi T), \qquad & \mbox{\textnormal{if} $\pi (i)<\pi (i+1)$}.
\end{cases}
\end{equation}

\end{proposition}

\begin{proof}
We prove this statement by induction in $n$.

\textit{Base case}: $n=0$. We have $f_0 (s \to s) =1$, and $f_0 (s \to \pi)=0$ for $\pi \ne s$. The statement is immediate.

\textit{Induction step}: Assume that the statement is proved for $n$, and let us prove it for $n+1$. Let $t_{n+1} = (A,A+1)$, and denote $x_{n+1}$ as $x$.
The definitions imply the following chain of equalities, where we use $\hat \pi := t_{n+1} \pi$:

\begin{multline*}
w_{t_{n+1}, x} \sum_{\pi \in S_N} f_n (s \to \pi) \pi = \sum_{\pi \in S_N} f_n (s \to \pi) \left( \mathbf{1}_{\pi^{(-1)}(A) < \pi^{(-1)}(A+1)} \left( (1-x) \pi + x \hat \pi \right) \right. \\ \left. + \mathbf{1}_{\pi^{(-1)}(A) > \pi^{(-1)}(A+1)} \left( (1-qx) \pi + q x \hat \pi \right)\right) = \sum_{\pi \in S_N} \pi \left( f_n (s \to \pi) (1-x) \mathbf{1}_{\pi^{(-1)}(A) < \pi^{(-1)}(A+1)} \right. \\ + \left. f_n (s \to \pi) (1-qx) \mathbf{1}_{\pi^{(-1)}(A) > \pi^{(-1)}(A+1)} + f_n (s \to \hat \pi) x \mathbf{1}_{\hat \pi^{(-1)}(A) < \hat \pi^{(-1)}(A+1)} \right. \\ + \left. f_n (s \to \hat \pi) q x \mathbf{1}_{\hat \pi^{(-1)}(A) > \hat \pi^{(-1)}(A+1)} \right).
\end{multline*}
Thus, we have
\begin{multline}
\label{eq:fn-rec}
f_{n+1} (s \to \pi) = f_n (s \to \pi) (1-x) \mathbf{1}_{\pi^{(-1)}(A) < \pi^{(-1)}(A+1)} + f_n (s \to \pi) (1-qx) \mathbf{1}_{\pi^{(-1)}(A) > \pi^{(-1)}(A+1)} \\ + f_n (s \to \hat \pi) x \mathbf{1}_{\hat \pi^{(-1)}(A) < \hat \pi^{(-1)}(A+1)}  + f_n (s \to \hat \pi) q x \mathbf{1}_{\hat \pi^{(-1)}(A) > \hat \pi^{(-1)}(A+1)} \\ = \mathbf{1}_{\pi^{(-1)}(A) < \pi^{(-1)}(A+1)} \left( f_n (s \to \pi) (1-x) + f_n (s \to \hat \pi) q x \right) \\ + \mathbf{1}_{\pi^{(-1)}(A) > \pi^{(-1)}(A+1)} \left( f_n (s \to \pi) (1-qx) + f_n (s \to \hat \pi) x \right).
\end{multline}
Exactly one of the two indicator functions in the last equation from \eqref{eq:fn-rec} is nonzero.
In order to check \eqref{eq:partAprop}, we need to track the expansions from \eqref{eq:fn-rec} of three terms $f_{n+1} \left( s T \to \pi \right)$, $f_{n+1} \left( s \to \pi T \right)$, and $f_{n+1} \left( s \to \pi \right)$. These expansions depend on the ordering of $(\pi^{(-1)} (A)$, $\pi^{(-1)} (A+1))$ and $((\pi T)^{(-1)} (A)$, $(\pi T)^{(-1)} (A+1))$.
After that, we will apply the induction hypothesis to the terms coming from the expansion of $f_{n+1} \left( s T \to \pi \right)$, and we also need to track which of two cases in \eqref{eq:partAprop} appears. For this we will need to know the ordering of $(\pi(i), \pi(i+1))$ and $(\hat \pi(i), \hat \pi(i+1))$.

Hence, we need to analyze several possible scenarios.

\smallskip

\textbf{Case 1:} Assume that $\{ A, A+1 \} \cap \{ \pi (i), \pi (i+1) \} = \varnothing$. Note that in this case $\pi^{-1} (A) = T \pi^{-1} (A)$, $\pi^{-1} (A+1) = T \pi^{-1} (A+1)$, $\pi(i) = \hat \pi(i)$, $\pi(i+1) = \hat \pi(i+1)$.

This implies that the three terms $f_{n+1} \left( s T \to \pi \right)$, $f_{n+1} \left( s \to \pi T \right)$, and $f_{n+1} \left( s \to \pi \right)$ have the same linear expansion in terms of $f_n$-functions. Moreover, by the induction hypothesis, the terms $f_n (s T \to \pi)$ and $f_n (s T \to \hat \pi)$ appearing from $f_{n+1} \left( s T \to \pi \right)$ are expressed in terms of $f_n (s \to \pi T)$, $f_n (s \to \pi)$ and $f_n (s \to \hat \pi T)$, $f_n (s \to \hat \pi)$, respectively, with the same coefficients. This gives the statement of the proposition.

\smallskip

\textbf{Case 2:} Assume that $|\{ A, A+1 \} \cap \{ \pi (i), \pi (i+1) \} | =1$.
While exact equalities as in Case 1 might not hold here, we claim that inequalities $\pi^{-1} (A) < \pi^{-1} (A+1)$ and $T \pi^{-1} (A) < T \pi^{-1} (A+1)$ are true or false simultaneously, and inequalities $\pi(i) < \pi(i+1)$ and $\hat \pi(i) < \hat \pi(i+1)$ are true or false simultaneously. Given this, we can claim, as in Case 1, that our expansions are the same and the statement reduces to the induction hypothesis.

In order to prove the claim about inequalities, note that in this case one of numbers $\pi^{-1} (A)$, $\pi^{-1} (A+1)$ is inside the set $\{ i, i+1 \}$, while the other is outside of this set. The application of $T$ does not change the latter one, and the former one is still inside the set $\{ i, i+1 \}$ after the application, which implies that  $\pi^{-1} (A) < \pi^{-1} (A+1)$ and $T \pi^{-1} (A) < T \pi^{-1} (A+1)$  are true or false simultaneously. Quite similarly, note that one of numbers $\pi(i), \pi(i+1)$ is inside the set $\{ A, A+1 \}$, while the other is outside of this set. The application of $(A,A+1)$ does not change the latter one, and the former one is still inside $(A,A+1)$. Thus, $\pi(i) < \pi(i+1)$ and $\hat \pi(i) < \hat \pi(i+1)$ are true or false simultaneously.

\smallskip

\textbf{Case 3a:} Assume that $A=\pi (i)$, $A+1=\pi (i+1)$. Then by a direct inspection one checks that $\pi^{-1} (A) < \pi^{-1} (A+1)$, $T \pi^{-1} (A) > T \pi^{-1} (A+1)$, 
and also $\hat \pi T = \pi$, $\pi T = \hat \pi$. Thus, from \eqref{eq:fn-rec} we have following expansions
\begin{align}
f_{n+1} (s \to \pi) = f_n (s \to \pi) (1-x) + f_n (s \to \hat \pi) qx, \qquad \\
\label{eq:cas3Arec}
f_{n+1} (s T \to \pi) = f_n (s T \to \pi) (1-x) + f_n (s T \to \hat \pi) qx, \\
f_{n+1} (s \to \pi T) = f_n (s \to \pi T) (1-q x) + f_n (s \to \hat \pi T) x.
\end{align}
Using inequalities $\pi (i) =A < A+1 = \pi (i+1)$, $\hat \pi (i) =A+1 > A = \hat \pi (i+1)$, by the induction hypothesis we have
\begin{equation*}
f_n ( s T \to \pi) = q f_n ( s \to \pi T), \qquad f_n ( s T \to \hat \pi) = f_n ( s \to \pi) + (1-q) f_n ( s \to \hat \pi).
\end{equation*}
Plugging this into \eqref{eq:cas3Arec} we get (using $\hat \pi = \pi T$)
\begin{multline*}
f_{n+1} (s T \to \pi) = (1-x) q f_n ( s \to \pi T) + qx \left( f_n ( s \to \pi) + (1-q) f_n ( s \to \hat \pi) \right) \\ = qx f_n ( s \to \pi) + (q - q^2 x) f_n ( s \to \pi T) = q f_{n+1} ( s \to \pi T),
\end{multline*}
as required.

\smallskip

\textbf{Case 3b:} Assume that $A=\pi (i+1)$, $A+1=\pi (i)$. This case is very similar to the previous one, yet we will give a full computation.
By a direct inspection one checks that $\pi^{-1} (A) > \pi^{-1} (A+1)$, $T \pi^{-1} (A) < T \pi^{-1} (A+1)$,
and also $\hat \pi T = \pi$, $\pi T = \hat \pi$. Thus, from \eqref{eq:fn-rec} we have following expansions
\begin{align}
f_{n+1} (s \to \pi) = f_n (s \to \pi) (1-q x) + f_n (s \to \hat \pi) x, \qquad \\
\label{eq:cas3Brec}
f_{n+1} (s T \to \pi) = f_n (s T \to \pi) (1-qx) + f_n (s T \to \hat \pi) x, \\
f_{n+1} (s \to \pi T) = f_n (s \to \pi T) (1-x) + f_n (s \to \hat \pi T) q x.
\end{align}
Using inequalities $\pi (i) =A+1 > A = \pi (i+1)$, $\hat \pi (i) =A < A+1 = \hat \pi (i+1)$, by the induction hypothesis we have
\begin{equation*}
f_n ( s T \to \pi) = f_n ( s \to \pi T) + (1-q) f_n ( s \to \pi) , \qquad f_n ( s T \to \hat \pi) = q f_n ( s \to \pi).
\end{equation*}
Plugging this into \eqref{eq:cas3Brec} we get
\begin{multline*}
f_{n+1} (s T \to \pi) = (1-qx) \left( f_n ( s \to \pi T) + (1-q) f_n ( s \to \pi) \right) + x q f_n ( s \to \pi) \\ = ( 1 - q + q^2 x) f_n ( s \to \pi) + (1-qx) f_n ( s \to \pi T) = f_{n+1} (s \to \pi T) + (1-q) f_{n+1} (s \to \pi).
\end{multline*}

We have exhausted all possible cases. Thus, the proposition is proved.


\end{proof}

\subsection{Symmetry theorem}

In this section we prove a somewhat surprising symmetry for coefficients $f_n \left( e \to \pi \right)$, where $e$ is the identity element of $S_N$. In addition to \eqref{eq:model-def}, let us define the coefficients $\{ \tilde f_n (s \to \pi) \}_{s,\pi \in S_N}$ via
\begin{equation}
\label{eq:model-def2}
w_{t_1, x_1} w_{t_{2}, x_{2}} \dots w_{t_n, x_n} s =: \sum_{\pi \in S_N} \tilde f_n (s \to \pi) \pi, \qquad s \in S_N.
\end{equation}
Note that the $w_{t,x}$ operators are applied in the opposite order here as compared to \eqref{eq:model-def}.

\begin{theorem}
\label{th1}
In notations and assumptions above, for any choice of transpositions $t_1, \dots, t_n$ and parameters $x_1, \dots, x_n$, we have
\begin{equation}
\label{eq:theor1}
f_n (e \to \pi) = \tilde f_n ( e \to \pi^{-1}).
\end{equation}
\end{theorem}

\begin{figure}
\begin{center}
\noindent{\scalebox{0.6}{\includegraphics[height=7cm]{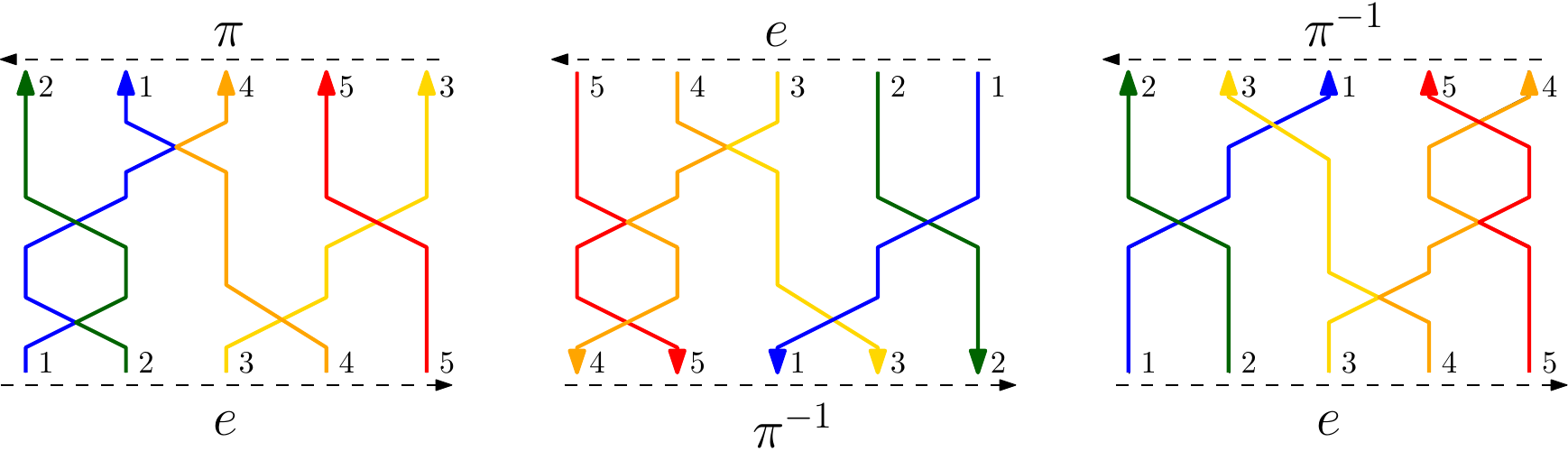}}}
\caption{The symmetry theorem. The middle picture is obtained from the leftmost one by changing directions of all arrows. The rightmost picture is the rotation of the middle picture by 180 degrees. \label{fig:5}}
\end{center}
\end{figure}

\begin{remark}
The quantities involved in the statement are illustrated in Figure \ref{fig:5}. Note that the statement has a nice pictorial interpretation: We just need to invert and rotate all arrows (or rather boundary conditions) on the picture.
\end{remark}

\begin{proof}

We prove the statement by induction in $n$. For $n=1$ the statement is tautological if $\pi$ is an identity or a transposition, and both sides of \eqref{eq:theor1} are equal to 0 for other $\pi$'s. Let us do $n \to n+1$ induction step. Let $t_{n+1} =: (A,A+1)$, denote $x_{n+1}$ as $x$, and let $\hat \pi := (A,A+1) \pi$.

Equation \eqref{eq:fn-rec} reads
\begin{multline}
\label{eq:fn-recDW}
f_{n+1} (e \to \pi) = \mathbf{1}_{\pi^{(-1)}(A) < \pi^{(-1)}(A+1)} \left( f_n (e \to \pi) (1-x) + f_n (e \to \hat \pi) q x \right) \\ + \mathbf{1}_{\pi^{(-1)}(A) > \pi^{(-1)}(A+1)} \left( f_n (e \to \pi) (1-qx) + f_n (e \to \hat \pi) x \right),
\end{multline}
which reduces $f_{n+1}$-coefficients to $f_n$'s. In order to do the same reduction for $\tilde f_{n+1}$, note that
\begin{equation*}
w_{t_{n+1}, x} e = (1- x) e + x (A,A+1).
\end{equation*}
Therefore, \eqref{eq:model-def2} implies
\begin{equation}
\label{eq:symm-backw1}
\tilde f_{n+1} \left( e \to \pi^{-1} \right) = (1-x) \tilde f_{n} \left( e \to \pi^{-1} \right) + x \tilde f_{n} \left( (A,A+1) \to \pi^{-1} \right).
\end{equation}
By Proposition \ref{prop:main} we have
\begin{multline*}
\tilde f_{n} \left( (A,A+1) \to \pi^{-1} \right) = \mathbf{1}_{\pi^{-1} (A) > \pi^{-1} (A+1)} \left( \tilde f_n (e \to \pi^{-1} (A,A+1)) + (1-q) \tilde f_n (e \to \pi^{-1}) \right) \\ + \mathbf{1}_{\pi^{-1} (A) < \pi^{-1} (A+1)} q \tilde f_n (e \to \pi^{-1} (A,A+1)).
\end{multline*}
Plugging this into \eqref{eq:symm-backw1}, we obtain
\begin{multline}
\label{eq:tildeFn-rec}
\tilde f_{n+1} \left( e \to \pi^{-1} \right) = \mathbf{1}_{\pi^{-1} (A) < \pi^{-1} (A+1)} \left( (1-x) \tilde f_{n} \left( e \to \pi^{-1} \right) + q x \tilde f_n (e \to \pi^{-1} (A,A+1)) \right) \\
+ \mathbf{1}_{\pi^{-1} (A) > \pi^{-1} (A+1)} \left( (1- q x) \tilde f_{n} \left( e \to \pi^{-1} \right) + x \tilde f_n (e \to \pi^{-1} (A,A+1)) \right).
\end{multline}
Comparing \eqref{eq:tildeFn-rec} with \eqref{eq:fn-recDW} and noting that $\hat \pi^{-1} = (A,A+1) \pi^{-1}$, we reduce the induction step to the induction hypothesis.

\end{proof}

\section{Multi-colored ASEP and stochastic six-vertex model}
\label{sec:symmAsepGen}

We start with a description of a common setting for a colored (or multi-species, or multi-type) version of the asymmetric simple exclusion process (ASEP) and the stochastic six vertex model (S6V). We consider an interacting particle system in which particles live on the integer lattice $\Z$ and each integer location contains exactly one particle. As parameters, we have an asymmetry parameter $q \in [0,1]$, and the set of colors which we will fix as $\Z \cup \{ +\infty \}$.

A particle configuration is a map $\eta: \Z \to \Z \cup \{ +\infty \}$, which can be regarded as the information that a particle at $z \in \Z$ has color $\eta(z)$. Let $\mathfrak C$ be the set of all configurations. For a transposition $(z,z+1)$ with $z, z+1 \in \Z$, let $\sigma_{(z,z+1)}: \mathfrak C \to \mathfrak C$, be a swap operator defined by
$$
(\sigma_{(z,z+1)} \eta) (i) =
\begin{cases}
\eta(i+1), \qquad & i=z, \\
\eta(i-1), \qquad & i=z+1, \\
\eta(i), \qquad & i \in \Z \backslash \{ z,z+1 \}.
\end{cases}
$$

Define a random \textit{asymmetric swap} $W_{(z,z+1),x}$, $z \in \Z$, $x \in [0;1]$, which assigns to each $\eta$ one of the two configurations $\eta$, $\sigma_{(z,z+1)} (\eta)$ with certain probabilities. Namely, if $\eta(z)<\eta(z+1)$, then $W_{(z,z+1),x} (\eta) = \eta$ with probability $1-x$, and $W_{(z,z+1),x} (\eta) = \sigma_{(z,z+1)} (\eta)$ with probability $x$. If $\eta(z) > \eta(z+1)$, then $W_{(z,z+1),x} (\eta) = \eta$ with probability $1-q x$, and $W_{(z,z+1),x} (\eta) = \sigma_{(z,z+1)} (\eta)$ with probability $q x$. See Figure \ref{Fig1-intro} for an illustration of $W_{(z,z+1),x}$. Finally, if $\eta(z) = \eta(z+1)$, then $W_{(z,z+1),x} (\eta) = \eta$ with probability 1.


In order to define an inhomogeneous continuous time ASEP, let us fix a collection of functions $\{ r(z,t) \}_{z \in \Z, t \in \R_{\ge 0}}$ as parameters of inhomogeneity. We assume that $0 \le r(z,t) \le const$ for a constant which does not depend on $z$ and $t$, and that $r(z,t)$ is a piecewise continuous function of $t$ for any fixed $z$.
Consider a collection of independent Poisson processes $\{ \mathcal P (z) \}_{z \in \Z}$, where $\mathcal P (z)$ has a state space $\R_{\ge 0}$ and rate $r(z,t)$.

Let $\eta_0$ be a (either deterministic or random) particle configuration which plays the role of an initial condition. We define a continuous-time stochastic evolution $\{ \eta_t \}_{t \in \R_{\ge 0}}$, $\eta(t) \in \mathfrak C$, by applying $W_{(z,z+1),1}$ to $\eta_{t-0}$ whenever we obtain a point from $\mathcal P (z)$. It is readily shown via standard techniques that under our assumptions such a random process is well-defined, see \cite{H1}, \cite{H2}, \cite{H3}, \cite{L4}.

In the ASEP case, we will be particularly interested in the following two processes. In the first one, we start with the initial configuration $\eta^{ASEP;\bm{gen}}_0 (z)=z$, $z \in \Z$. Next, we apply the asymmetric swaps $W_{s_1,1}$, $W_{s_2,1}$, ..., $W_{s_k,1}$, where $\{s_i \}_{i=1}^k$ is a finite sequence of transpositions of neighboring integers ($W_{s_1,1}$ is the first swap to be applied). After it, we start the continuous time process described in the previous paragraph with general inhomogeneous functions $\{ r_{\bm{gen}} (z,t) \}$. Denote by $\eta^{ASEP;\bm{gen}}_{s_1, \dots, s_k;t}$ the (random) configuration after time $t$ (note that the random configuration also depends on $q$ and $\{ r_{\bm{gen}} (z,t) \}$, but we omit this in notations).

Next, we want to construct the second process with the use of the same data $s_1, \dots, s_k, \{ r_{\bm{gen}} (z,t) \}$.
Let us fix a positive real $\bm{\tau}$. Let us time-reflect inhomogeneity parameters on $\left[ 0;\tau \right]$: Set $\hat r_{\bm{gen}} (z,t) :=  r_{\bm{gen}} (z, \bm{\tau}-t)$, $0 \le t \le \bm{\tau}$.
Then, take the initial configuration $\hat \eta^{ASEP;\bm{gen};\bm{\tau}}_0 (z)=z$, $z \in \Z$, and start a continuous time process from it with the use of the new inhomogeneity parameters $\{ \hat r_{\bm{gen}} (z,t) \}$. After time $t \le \bm{\tau}$, we apply to the resulting (random) configuration the asymmetric swaps in the order $W_{s_k,1}$, $W_{s_{k-1},1}$, ..., $W_{s_1,1}$. Denote by $\hat \eta_{t; s_k, \dots, s_1}^{ASEP;\bm{gen};\bm{\tau}}$ the resulting (random) configuration. Note that $\hat \eta_{t; s_k, \dots, s_1}^{ASEP;\bm{gen};\bm{\tau}}$ is a (random) bijection between integers, since both colors and positions are labelled by integers and there is exactly one particle colored by any given color in our initial configuration. Thus, it is possible to define $\mathrm{inv}\left(\hat \eta_{t; s_k, \dots, s_1}^{ASEP;\bm{gen};\bm{\tau}} \right)$ as the inverse of this bijection.

The process $\hat \eta_{t; s_k, \dots, s_1}^{ASEP;\bm{gen};\bm{\tau}}$ is just the time reversion of the process $\eta^{ASEP;\bm{gen}}_{s_1, \dots, s_k;t}$: We reverse in time the inhomogeneity parameters and apply the finite collection of swaps in the end instead of the beginning. It turns out that this change of the time direction leads to the change of roles between colors and positions.

\begin{theorem}
\label{th:genASEP}
For any $s_1, \dots, s_k$, $t$, $q$, and $r_{\bm{gen}}(z,t)$ satisfying assumptions above, the configuration $\eta^{ASEP;\bm{gen}}_{s_1, \dots, s_k;t}$ has the same distribution as $\mathrm{inv}\left( \hat \eta_{t; s_k, \dots, s_1}^{ASEP;\bm{gen};t} \right)$.

In other words, the distribution of colors of particles that are placed in positions $I \subset \Z$ after $\eta^{ASEP;\bm{gen}}_{s_1, \dots, s_k;t}$ is the same as the distribution of positions of particles with colors from $I$ that underwent $\hat \eta_{t; s_k, \dots, s_1}^{ASEP;\bm{gen};t}$.
\end{theorem}

\begin{proof}
The argument is similar to \cite[Proof of Theorem 1.4]{AmAnV}.

Note that for any $z \in \Z$, there is a probability $p>0$ such that with probability $p$ no asymmetric swaps were applied to $z, z+1$ (here we use that all the rate functions are bounded by the same constant). Further, these probabilities are the same for the direct and reversed processes.
Each such $z$ splits $\Z$ into two parts which do not interact with each other during the process.
Therefore, with probability 1 the infinite permutation $\eta^{ASEP;\bm{gen}}_{s_1, \dots, s_k;t}$ is a product of finite commuting permutations that act on disjoint subsets of $\Z$, and for any such finite permutation the claim follows from Theorem \ref{th1} and time-reflection symmetry of the inhomogeneous one dimensional Poisson process.

\end{proof}

Let us formulate a version of this theorem for a S6V (stochastic six vertex) model. As inhomogeneity parameters, we consider the set $\{ x(z,T) \}$ where $z \in \Z$ is viewed as position, $T \in \Z_{\ge 0}$ is viewed as discrete time, the parameter $x(z,T)$ is relevant for the process iff $z \equiv T (\mbox{mod } 2)$, and $0 \le x(z,T) <1$.

Let $\eta(0)$ be (either deterministic or random) particle configuration which plays the role of an initial condition of the S6V model. We define the discrete-time stochastic process $\{ \eta(T) \}_{T \in \Z_{\ge 0}}$, $\eta(T) \in \mathfrak C$, inductively, by setting $\eta(T) = \prod_{z \in 2 \Z} W_{z,x(z,T)} \eta(T-1)$ for even $T$, and $\eta(T) = \prod_{z \in 2 \Z+1} W_{z,x(z,T)} \eta(T-1)$ for odd $T$. Note that the set $\{ W_{z,x(z,T)} \}_{z \in 2 \Z}$ (and, analogously, the set $\{ W_{z,x(z,T)} \}_{z \in 2 \Z+1}$) consists of mutually commuting operators, so we can apply them in any order.

As before, we will be particularly interested in two processes. In the first one, we start with an initial configuration $\eta^{6V;\bm{gen}}_0 (z)=z$, $z \in \Z$. Next, we apply the asymmetric swaps $W_{s_1,1}$, $W_{s_2,1}$, ..., $W_{s_k,1}$, where $\{s_i \}_{i=1}^k$ is a finite sequence of
transpositions of neighboring integers. After it, we start the discrete time process described in the previous paragraph. Denote by $\eta^{6V;\bm{gen}}_{s_1, \dots, s_k;T}$ the (random) configuration after time $T$. For convenience, we will assume that $T$ is odd (though the assumption can be easily omitted if needed).

For the second process we consider the same functions $\{ x(z,T) \}$ and integers $s_1, \dots, s_k$ as for the first process. Let us fix an odd $\bm{\tau} \in \Z_{>0}$, and consider parameters $ \hat x(z,\bm{\tau}) := x(z,\bm{\tau} - T)$, $0 \le T \le \bm{\tau}$, which correspond to the time reversion. We start the process with an initial configuration $\hat \eta_{T; s_k, \dots, s_1}^{6V;\bm{gen};\bm{\tau}} (z)=z$, $z \in \Z$. Next, we run the discrete time process with reversed time inhomogeneity parameters. After the time $\bm{\tau}$ we apply to the resulting configuration the asymmetric swaps $W_{s_k,1}$, $W_{s_{k-1},1}$, ..., $W_{s_1,1}$. Let us denote by $\hat \eta_{\tau; s_k, \dots, s_1}^{6V;\bm{gen};\bm{\tau}}$ the (random) configuration that we obtain after that.

\begin{theorem}
\label{th:genS6V}
For any $s_1, \dots, s_k$, $T$, $q$, and $\{ x(z,T) \}$ satisfying assumptions above, the configuration $\eta^{6V;\bm{gen}}_{s_1, \dots, s_k;T}$ has the same distribution as $\mathrm{inv} \left( \hat \eta_{T; s_k, \dots, s_1}^{6V;\bm{gen};T} \right)$.
\end{theorem}

\begin{proof}
We required that $x(z,T)<1$; therefore, any asymmetric swap applied to any configuration does not change it with positive probability. Thus, for any $z \in \Z$ there is a probability $p>0$ such that with probability $p$ the particles with colors $\le z$ stay at positions $ \le z$ during the evolution. The rest of the proof is analogous to Theorem \ref{th:genASEP}.
\end{proof}

\begin{remark}
Theorems \ref{th:genASEP} and \ref{th:genS6V} involve two processes which are generally quite different because of the change of inhomogeneity parameters. However, if we assume, for example, in Theorem \ref{th:genASEP} that $r(z,t)$ does not depend on $t$ for all $z \in \Z$, then the processes will coincide --- they are both the ASEP on $\Z$ with the same space inhomogeneities.
\end{remark}

\begin{remark}

Theorems \ref{th:genASEP} and \ref{th:genS6V} involve a combination of finitely many deterministic swaps and the usual ASEP / S6V process after them. It is easy to formulate similar, but more general statements with basically the same proofs. For example, one can perform deterministic swaps not only in the beginning / end, but also in several places in the middle; one can consider discrete time sequential or parallel updates instead of continuous time in ASEP; one can consider higher spin stochastic vertex models instead of S6V (via fusion), etc. We restricted ourselves to the versions which are sufficient for applications below.
\end{remark}

\section{Second class particles in inhomogeneous ASEP}

The most studied versions of the multi-colored ASEP appear when the particles are colored in two or three colors only. We will use the sets of colors $\{1, +\infty \}$ for the former case, and $\{1,2, +\infty \}$ for the latter case. In such a situation one typically addresses particles with color $1$ as the first class particles, particles with color $2$ as the second class particles, and particles with color $+\infty$ as holes. Let us show now that Theorem \ref{th:genASEP} (which involves infinitely many colors in an essential way) allows to obtain new facts for these models.

All our applications of Theorems \ref{th:genASEP} and \ref{th:genS6V} will use the same simple idea which allows to reduce the number of colors in the system. Instead of distinguishing infinitely many colors from $\Z$, let us choose arbitrary integers $\bm{A} \le \bm{B}$ and consider particles with colors $\le \bm{A}$ as the first class particles, particles with colors strictly between $\bm{A}$ and $\bm{B}$ as the second class particles, and particles with colors $\ge \bm{B}$ as holes (see Figure \ref{fig:7}). As readily visible from definitions, this provides a coupling between a $\Z$-color and a three-color (or two-color if $\bm{A}=\bm{B}$) process under which the set occupied by the first class particles evolves in the same way as the set occupied by the particles with colors $\le \bm{A}$, and the analogous property holds for the second class particles and holes.

\begin{figure}
\begin{center}
\noindent{\scalebox{0.2}{\includegraphics[height=7cm]{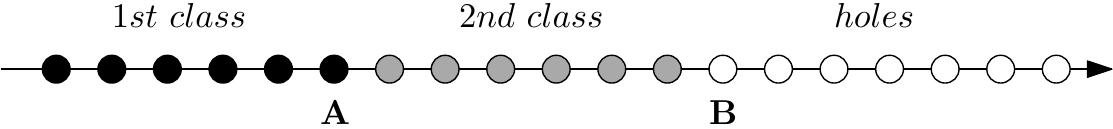}}}
\caption{ Initial conditions of a three color process obtained by ``forgetting'' a part of information about a multi-color process.
 \label{fig:7}}
\end{center}
\end{figure}

In this section we address the behavior of a single second class particle in the three-color inhomogeneous ASEP started with the step initial condition.

Let us define two continuous time ASEP processes.
As the first one, consider a three-color ASEP with the initial configuration
\begin{equation}
\eta_0^{2;inh} (z) = \begin{cases}
1, \qquad & z <0, \\
2, \qquad & z=0, \\
+\infty, \qquad & z> 0.
\end{cases}
\end{equation}
As inhomogeneity parameters, we choose arbitrary functions $\{ r_{inh} (z,t) \}$ satisfying general conditions of Section \ref{sec:symmAsepGen}. 
We are interested in the behavior of the unique second class particle in the inhomogeneous continuous time ASEP process. Denote its position at time $t$ as $\mathfrak f_2 (t)$.

For the second ASEP process, we fix $x \in \R$ and $\tau \in \R_{\ge 0}$, and consider a two-color ASEP with the (shifted) step initial condition
\begin{equation}
\eta_0^{inh;x;\tau} (z) = \begin{cases}
1, \qquad & z \le x, \\
+\infty, \qquad & z> x,
\end{cases}
\end{equation}
and inhomogeneity parameters $\hat r_{inh} (z,t) := r_{inh} (z,\tau-t)$. This process is defined up to time $\tau$.
Let $\mathfrak{p}_1^x (\tau) := \mathrm{Prob} (\hat \eta_{\tau}^{inh;x;\tau} (0) = 1)$ be the density function of the (first class) particles in this process at position 0 and time $\tau$.

\begin{theorem}
\label{th:inh2cl}
In the notations above, for any $\tau \ge 0$ we have
\begin{equation}
\label{eq:inh2cl}
\mathrm{Prob} ( \mathfrak f_2 (\tau) \le x) = \mathfrak{p}_1^x (\tau).
\end{equation}
\end{theorem}

\begin{proof}
Let us consider a $\Z$-colored ASEP process $\eta^{ASEP;\bm{gen}}_{\varnothing;t}$ with no deterministic swaps ($k=0$) and inhomogeneity parameters $\{ r_{inh} (z,t) \}$, and let $\hat \eta_{t; \varnothing}^{ASEP;\bm{gen};\tau}$ be its reversed process (see Section \ref{sec:symmAsepGen} for definitions). By Theorem \ref{th:genASEP}, $\eta^{ASEP;\bm{gen}}_{\varnothing;\tau}$ has the same distribution as $ \mathrm{inv} \left( \hat \eta_{\tau; \varnothing}^{ASEP;\bm{gen};\tau} \right)$. In particular,
\begin{equation}
\label{eq:234}
\mathrm{Prob} \left( \mathrm{inv} \left( \eta^{ASEP;\bm{gen}}_{\varnothing;\tau} \right) (0) \le x \right) = \mathrm{Prob} \left( \hat \eta_{\tau; \varnothing}^{ASEP;\bm{gen};\tau} (0) \le x \right).
\end{equation}
We claim that this equality is equivalent to \eqref{eq:inh2cl}. Indeed, for the process $\eta^{A;\bm{gen}}_{\varnothing;\tau}$ consider all particles of color $<0$ as the first class particles, the particle of color 0 as the second class particle, and particles of color $>0$ as holes. This creates a coupling between $\eta^{A;\bm{gen}}_{\varnothing;\tau}$ and $\eta_t^{2;inh}$ under which the evolution of the particle of color 0 from the former one coincides with the evolution of the second class particle from the latter one. Thus, the left-hand side of \eqref{eq:234} coincides with the left-hand side of \eqref{eq:inh2cl}. In a similar fashion, for $\hat \eta_{t; \varnothing}^{A;\bm{gen};\tau} $ consider all particles of color $\le x$ as the first class particles, and particles of color $>x$ as holes. This creates a coupling between $\hat \eta_{t; \varnothing}^{A;\bm{gen};\tau} $ and $\eta_t^{inh;x;\tau}$, under which the evolution of the set of positions filled by particles of color $\le x$ / the first class particles coincide. Thus, the right-hand side of \eqref{eq:234} coincides with the right-hand side of \eqref{eq:inh2cl}, which completes the proof.

\end{proof}

\begin{remark}
Theorem \ref{th:inh2cl} relates the distribution of the second class particle with the density of an inhomogeneous ASEP with one class of particles only. The asymptotic behavior of the latter object is one of the main questions in the theory of interacting particle systems, and it is known that under quite general conditions the asymptotic density follows the entropy solution of the inviscid Burgers equation (see \cite{B}, \cite{GKS}). Theorem \ref{th:inh2cl} translates these results into the asymptotic behavior of the second class particle.

It is important to stress that for nontrivial inhomogeneities Theorem \ref{th:inh2cl} involves two substantially \textit{different} processes. We do not know how to relate the distribution of the second class particle to the density properties of the process itself, without reversing the time.
\end{remark}

\begin{remark}
Using arguments from Section \ref{sec:finPert} below, Theorem \ref{th:inh2cl} can be extended to the case of a finite perturbation of the step initial condition. Then the right-hand side of \eqref{eq:inh2cl} will contain higher correlation functions of a two-color process instead of the first correlation function only. From the asymptotic point of view, this allows to establish the asymptotic behavior of second class particles whenever we know the local equilibrium type of results for inhomogeneities under investigation.
\end{remark}

\section{Second class particles in a finite perturbation of step initial conditions}
\label{sec:finPert}



In this section we study a homogeneous continuous time ASEP: All rate functions are set to be $r(z,t) \equiv 1$. Let us start with the following foundational result about the two-color ASEP. Consider the step initial condition $\eta_0^{st} (z) =1$ for $z \le 0$, and $\eta_0 (z) =+\infty$ for $z > 0$, (see Figure \ref{fig:8}) and denote the configuration at time $t$ by $\eta_t^{st}$.

\begin{figure}
\begin{center}
\includegraphics[width=10cm]{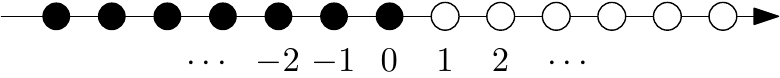}
\caption{ The step initial condition.
 \label{fig:8}}
\end{center}
\end{figure}

\begin{theorem}
\label{th:asepLLN}
Let $m= m(t)$, $t \in \R_{\ge 0}$, be a collection of integers such that
$$
\lim_{t \to \infty} \frac{m(t)}{t} = y, \qquad y \in \R.
$$
Then
\begin{equation}
\label{eq:dens1}
\lim_{t \to \infty} P ( \eta_t^{st} (m(t)) =1) = d(y) :=
\begin{cases}
0, \qquad & y \ge (1-q), \\
\frac{1}{2} \left( 1 - \frac{y}{1-q} \right), \qquad & -(1-q) < y < (1-q), \\
1, \qquad & y \le -(1-q).
\end{cases}
\end{equation}
Moreover, for any fixed $L \in \Z_{>0}$ the random variables $\{ \eta_t (m(t)+i) \}_{i=-L,..,L}$ converge, as $t \to \infty$, to i.i.d. Bernoulli distributions with probability of 1 equal to $d(y)$.
\end{theorem}
Theorem \ref{th:asepLLN} is well known (it goes back to \cite{AV}, \cite{BeF}, see also earlier works \cite{L1}, \cite{L2} for $y=0$ case, and \cite{R} for $q=0$ case). The first claim is typically referred to as the hydrodynamic limit (applied to a particular initial condition), the second claim is referred to as the local equilibrium.

Let $L$ be an arbitrary positive integer, and let
$$
I: \{ -L, -L+1, \dots, -1,0,1, \dots, L-1, L \} \to \{1,2, +\infty\}
$$
be an arbitrary fixed map between these sets such that $I^{-1} (2)$ is non-empty. We consider the three-color ASEP with the initial condition
\begin{equation}
\label{eq:pIC}
\eta_0^{A;pert} (z) = \begin{cases}
1, \qquad & z <-L, \\
I(z), \qquad & -L \le z \le L, \\
+\infty, \qquad & z> L.
\end{cases}
\end{equation}

\begin{figure}
\begin{center}
\noindent{\scalebox{0.2}{\includegraphics[height=7cm]{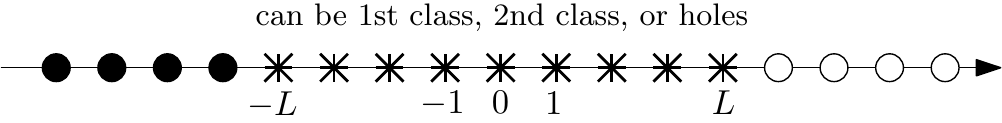}}}
\caption{ A finite perturbation of the step initial condition.
 \label{fig:9}}
\end{center}
\end{figure}

We refer to such $\eta_0^{A;pert}$ as a \textit{finite perturbation} of the step initial condition (see Figure \ref{fig:9}). Let $M := |I^{-1} (1)|$, and $N := |I^{-1} (2)|$. Thus, $N \ge 1$ is the (finite) number of the second class particles in our process.
It is natural to distinguish the leftmost, the second from the left, etc. second class particles. Let $S_1(t) \le S_2(t) \le \dots \le S_N (t)$ be the positions of the second class particles at time $t$. In other words, the rules of the evolution of the process give us the set
$\left\{ \left( \eta_t^{A;pert} \right)^{-1} (2) \right\} = \{ S_i (t) \} $, in which we number the elements in the increasing order (here by $\left( \cdot \right)^{-1}$ we denote the pre-image).
The main result of this section gives the asymptotic distribution of $S_k (t)$, for any $k$. In order to describe it, we need a bit more notation.

Consider a permutation 
$$
\pi_I: \{-L, \dots, L \} \to \{-L, \dots, L \}
$$ 
such that $\pi_I \left( \{ -L, \dots, -L+M-1 \} \right) = I^{-1} (1)$ and $\pi_I \left( \{ -L+M, \dots, -L+M+N-1 \} \right) = I^{-1} (2)$. There may be more than one such permutation, any choice will suit our purposes. Let $\pi_I = p_s p_{s-1} \dots p_2 p_1$ be its minimal decomposition into the product of transpositions of neighboring elements. Again, there might be more than one such decomposition, we fix an arbitrary one.

Fix a positive real $\bm{p} \le 1$, and consider a filling of $\{ -L, \dots, L \}$ by particles via $\tilde \eta (i) = 1$ with probability $\bm{p}$, $\tilde \eta (i) = +\infty$ with probability $1-\bm{p}$, and $\{ \tilde \eta (i) \}_{i=-L}^L$ are jointly independent. In words, we fill all the positions by particles independently and with probability $\bm{p}$. Next, we apply to this filling the asymmetric swap $W_{p_s,1}$, then $W_{p_{s-1},1}$, etc., and finally $W_{p_1,1}$. As a result, we obtain a certain (random) configuration of particles inside $\{-L, \dots, L\}$. Let $f_k (\bm{p},I)$ be the probability that the set $\{ -L+M, \dots, -L+M+N-1 \}$ contains at least $k$ particles of thus obtained configuration, $k=0, 1, \dots, N$.

\begin{theorem}
\label{th:finPert}

Consider a three-color ASEP started from a perturbed step initial condition of the form \eqref{eq:pIC}. In the notations above, the positions of the second class particles $S_1(t) \le S_2(t) \le \dots \le S_N (t)$ satisfy

\begin{equation}
\label{eq:theorAsep}
\lim_{t \to \infty} \mathrm{Prob} \left( \frac{S_k (t)}{t} < x \right) = f_{k} ( d(-x), I), \qquad x \in \R, \ \ k=1, \dots, N,
\end{equation}
where $d(x)$ is given by \eqref{eq:dens1}.
\end{theorem}

\begin{remark}
The quantity $f_k ( d(-x), I)$ can be straightforwardly computed. Explicit computations in some cases are given in Examples \ref{ex:basic}--\ref{ex:last} below.
\end{remark}

\begin{proof}

Consider a homogeneous $\Z$-color ASEP $\eta^{ASEP;\bm{hom}}_{p_1, \dots, p_s;t}$ as defined in Section \ref{sec:symmAsepGen} (we set $r(z,t)$ to be equal to 1 for all $z,t$ and write $\bm{hom}$ in notation because of this). Note that the definition and the minimality of the decomposition of $\pi_I$ imply that all initial asymmetric swaps act on a configuration in a deterministic way. Therefore, if we regard particles of colors $<(-L+M)$ as first class particles, particles of color weakly between $(-L+M)$ and $(-L+M+N-1)$ as second class particles, and all other colors as holes, we recover the initial configuration \eqref{eq:pIC}. This provides a coupling between the multi-color process and the three-color process $\eta_t^{A;pert}$ that we study.

For $y \in \R$, let $\mathcal L_{t;y}$ be the (random) set of integers $a$ such that $a\in \{-L+M, \dots ,-L+M+N-1\}$ and $\left( \eta^{ASEP;\bm{hom}}_{p_1, \dots, p_s;t} \right)^{-1} (a) < y$.

By the coupling between the processes, we have
\begin{equation}
\label{eq:111a}
\mathrm{Prob} \left( S_k (t) < y \right) = \mathrm{Prob} \left( |\mathcal L_{t;y}| \ge k \right).
\end{equation}

Now let us consider the process $\eta^{ASEP;\bm{hom}}_{t;p_s, \dots, p_1}$ (note the ordering change for $p_i$'s), which is related to $\eta^{ASEP;\bm{hom}}_{p_1, \dots, p_s;t}$ by the time reversion. Let us denote the (random) set of integers $\{ a: a \in \{(-L+M), \dots, (-L+M+N-1) \}$, $\eta^{ASEP;\bm{hom}}_{t;p_s, \dots, p_1} (a) <y \}$ by $\mathcal R_{t;y}$. Theorem \ref{th:genASEP} implies
\begin{equation}
\label{eq:111b}
\mathrm{Prob} \left( |\mathcal L_{t;y}| \ge k \right) = \mathrm{Prob} \left( | \mathcal R_{t;y} | \ge k \right).
\end{equation}
This was the key step, since the right-hand side involves the statement which depends only on the distinction between the colors that are less than $y$, and those that are greater or equal to $y$. Hence, this quantity can be determined via a standard two-color ASEP.

Let $\mu_t (z)$ be the configuration which appears in $\eta^{ASEP;\bm{hom}}_{t;p_s, \dots, p_1}$ after time $t$ but before the deterministic swaps. By Theorem \ref{th:asepLLN}, the variables
$$
\mathbf 1_{\mu_t (-L) <xt}, \mathbf 1_{\mu_t (-L+1) <xt}, \dots, \mathbf 1_{\mu_t (L) <xt}
$$
are asymptotically independent and identically distributed as $t \to \infty$, with their distributions given by
$$
\lim_{t \to \infty} \mathrm{Prob} \left( \mathbf 1_{\mu_t (l) < xt} =1 \right) = d(-x), \qquad -L \le l \le L, \ \  x \in \R.
$$
This produces the i.i.d Bernoulli filling of the segment $\{-L, -L+1, \dots, L\}$, which is present in the right-hand side of \eqref{eq:theorAsep}. Combining \eqref{eq:111a}, \eqref{eq:111b} and the application of the swaps $p_s$, $p_{s-1}$, ..., $p_1$ in this order leads to the quantities $f_k ( d(-x), I)$, which concludes the proof.

\end{proof}

We proceed to give some examples of Theorem \ref{th:finPert}.

\begin{example}
\label{ex:basic}
$L=0$, $I(0)=2$. Then
$$
\lim_{t \to \infty} \mathrm{Prob} \left( \frac{S_1 (t)}{t} < x \right) = d(-x).
$$
Since $d'(-x)$ is constant in $(-(1-q);1-q)$ and is equal to $0$ outside this interval, $S_1(t)$ is uniformly distributed on $(-(1-q);1-q)$. The result was obtained in \cite{FK}, \cite{FGM}.
\end{example}

See Fig. \ref{fig:examples} for a depiction of the initial conditions from the following three examples.

\begin{example}
\label{ex:11}
$L \in \Z_{\ge 0}$, $I(a)=2$, for $-L \le a \le 0$, and $I(a)=+\infty$, for $1 \le a \le L$. As in the previous example, $\pi_I$ is the identity. We have
\begin{multline*}
\lim_{t \to \infty} \mathrm{Prob} \left( \frac{S_k (t)}{t} < x \right) \\ = \sum_{l=k}^{L+1} \binom{L+1}{l} d(-x)^l (1-d(-x))^{L+1-l} = 1- \sum_{l=0}^{k-1} \binom{L+1}{l} d(-x)^l (1-d(-x))^{L+1-l}.
\end{multline*}
For $k=1$ this result was recently obtained in \cite{GSZ}. Note that the limiting distribution of $S_k$ can be obtained in the following way: we take $L+1$ uniform random variables on the segment $[-(1-q);(1-q)]$, and take the $k$-th smallest of them. The appearance of the uniform distribution might be attributed to the situation of Example \ref{ex:basic}. However, note that the \textit{joint} limiting distribution of $\{ S_k \}$ (which we do not address in the text) is known to be much more involved than them being independent even in the case of the TASEP, see \cite{AmAnV}.
\end{example}

\begin{example}
\label{ex:quadratic}
$L=1$, $I(-1)=2$, $I(0)=1$, $I(1)=+\infty$. In this case $\pi_I$ is equal to the transposition $(-1,0)$.
The behavior of the second class particle is given by
\begin{equation}
\label{eq:smPert}
\lim_{t \to \infty} \mathrm{Prob} \left( \frac{S_1 (t)}{t} < x \right) = d(-x) + (1-q) d(-x) (1-d(-x)).
\end{equation}
Indeed, our recipe works in the following way here. We consider sites $-1$ and $0$ filled independently with probability $d(-x)$ each, apply the asymmetric swap $W_{(-1,0),1}$, and we are interested in the probability that $0$ is filled by a particle after it. If $-1$ is filled by a particle before the swap, then 0 must be filled by a particle after the swap --- this gives the first term in the right-hand side of \eqref{eq:smPert}. If $-1$ is not filled by a particle before the swap, then $0$ should be filled, and the swap should not move this particle from 0 --- this gives the second term.

Note that this distribution significantly depends on $q$. Also note that this result contradicts \cite[Conjecture 1.2]{GSZ}.
\end{example}

\begin{figure}
\begin{center}
 {\scalebox{1.2}{\includegraphics{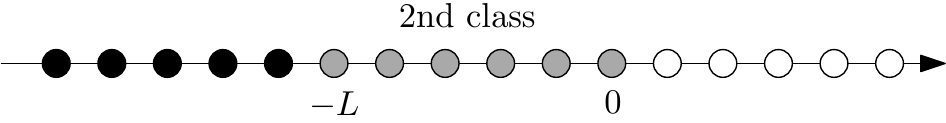}}} \\ \medskip \medskip
 {\scalebox{1.2}{\includegraphics{figure11.pdf}}} \\ \medskip \medskip
 {\scalebox{1.2}{\includegraphics{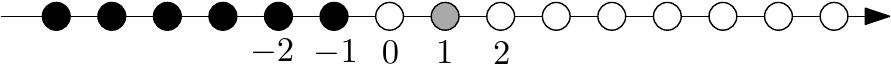}}}
 \caption{Initial conditions in Examples \ref{ex:11}, \ref{ex:quadratic}, and \ref{ex:33}, respectively.
 \label{fig:examples} }
\end{center}
\end{figure}

\begin{example}
\label{ex:33}
$L=1$, $I(1)=2$, $I(0)=+\infty$, $I(-1)=1$. In this case $\pi_I$ is equal to the transposition $(0,1)$.
The behavior of the second class particle is given by
\begin{equation}
\label{eq:smPert}
\lim_{t \to \infty} \mathrm{Prob} \left( \frac{S_1 (t)}{t} < x \right) = d(-x)^2 + q d(-x) (1-d(-x)).
\end{equation}
The computation is analogous to the previous example.
\end{example}

\begin{example}
\label{ex:last}
Let $q=0$ (so we are in the TASEP case), $I(a)=+\infty$, for $a=-L, \dots, -1$, $I(0)=2$, $I(a)=1$, for $a=1,\dots, L$. Then $\pi_I$ can be taken as the permutation $\pi_I (-z)=z$, for $z=-L, -L+1, \dots, L$ (see Fig. \ref{fig:GUE-GUE} for a depiction).
The behavior of the second class particle is given by
$$
\lim_{t \to \infty} \mathrm{Prob} \left( \frac{S_1 (t)}{t} < x \right) = \sum_{l=L+1}^{2L+1} d(-x)^l (1- d(-x))^{2L+1-l} \binom{2L+1}{l}.
$$
Indeed, note that in the TASEP case the application of $\pi_I^{-1} = \pi_I$ moves all particles in $\{-L, \dots, L\}$ to the right-most possible positions. Thus, after such shift the position $0$ is occupied by a particle if and only if there are at least $L+1$ particles initially.

In the case of TASEP and one second class particle, a general way to find the asymptotic distribution of the second class particle was found by \cite{CP}.
\end{example}

\begin{figure}
\begin{center}
\includegraphics[width=15cm]{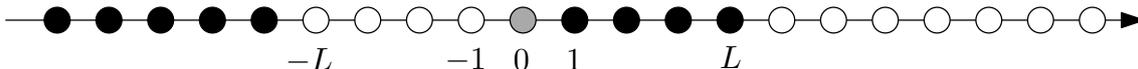}
\caption{ A GUE-GUE shock type of the initial condition.
 \label{fig:GUE-GUE}}
\end{center}
\end{figure}

\begin{remark}
The analysis of this section can be applied to the homogeneous S6V model with the use of Theorem \ref{th:genS6V} instead of Theorem \ref{th:genASEP}. The only subtle point is that, up to our knowledge, the analog of Theorem \ref{th:asepLLN} for S6V is not explicitly present in the current literature. However, it is plausible that such an analog holds with the density function found in \cite{BCG} and equilibrium measures studied in \cite{Ag1}. In fact, there are two approaches that might allow to get the result. First, a recent paper \cite{Ag2} establishes the local equilibrium for general initial conditions for the stochastic six vertex model on a cylinder. Very likely, the same methods work in the case of integer lattice as well. Second, for the step initial condition there are explicit formulas for certain observables, see \cite{BCG}, \cite{BM}. They might also allow to access necessary asymptotics.

Given such an analog of Theorem \ref{th:asepLLN}, our subsequent analysis of second class particles holds word by word. Note that limiting distributions will be different from the ASEP, since both the hydrodynamic behavior and the local equilibrium for S6V differ from those for the ASEP.
\end{remark}

\section{Second class particle in GUE-GUE shock}

In this section we will study a homogeneous ($r(z,t) \equiv 1$) three-color TASEP ($q=0$) with a particular initial condition. One new feature of our analysis is that we allow the initial condition to depend on time.

Let $L$ be a positive integer, and consider a TASEP denoted as $\eta_t^{2;tasep} (z)$ with the initial condition
$$
\eta_0^{2;tasep} (z) = \begin{cases}
1, \qquad & z <-L, \\
+\infty, \qquad & -L \le z \le -1, \\
2, \qquad & z=0, \\
1, \qquad & 1 \le z \le L, \\
+\infty, \qquad & z> L.
\end{cases}
$$

Let $\mathfrak{f^{tasep}} (t)$ be the position of the unique second class particle in the process.
Again, we will study the distribution of $\mathfrak{f^{tasep}} (t)$ by relating the process with a simpler process via the color-position symmetry. Consider a two-color TASEP $\eta_t^{tasep} (z)$ with the step initial condition
$$
\eta_0^{tasep} (z) = \begin{cases}
1, \qquad & z \le 0, \\
+\infty, \qquad & z >0.
\end{cases}
$$
Let $h^{tasep} (x,t)$ be the number of particles that are weakly to the right of $x$ in $\eta_t^{tasep} (z)$, for any $x \in \R$.

\begin{proposition}
\label{prop:shock}
For any $x \in \Z$ we have
\begin{equation}
\label{eq:SchockExact}
\mathrm{Prob} \left( \mathfrak{f^{tasep}} (t) \le x \right) = \mathrm{Prob} \left( h^{tasep} (-x-L,t) - h^{tasep} (-x+L,t) \ge L+1 \right).
\end{equation}
\end{proposition}

\begin{proof}
Let $\pi$ be a permutation of the set $\{-L, -L+1, \dots, L-1,L \}$ such that $\pi(-L+i) = L-i$, for $i=0,1,\dots, 2 L$. Let $\pi = s_{m} s_{m-1} \dots s_2 s_1$ be a minimal length decomposition of $\pi$ into transpositions of neighboring elements (there are many such decompositions, we choose any of them; we always have $m=2L(2L+1)/2$).

Consider a homogeneous $\Z$-color TASEP $\eta^{TASEP;\bm{hom}}_{s_1, \dots, s_m;t}$ and its time-reversed version $\hat \eta^{TASEP;\bm{hom}}_{t; s_m, \dots, s_1}$. By Theorem \ref{th:genASEP}, we know that
\begin{equation}
\label{eq:238}
\mathrm{Prob} \left( \mathrm{inv} \left( \eta^{TASEP;\bm{hom}}_{s_1, \dots, s_m;t} \right) (0) \le x \right) = \mathrm{Prob} \left( \hat \eta^{TASEP;\bm{hom}}_{t; s_m, \dots, s_1} (0) \le x \right).
\end{equation}
Treating colors $<0$ as the first class particles, color $0$ as the second class particle, and colors $>0$ as holes, we see that the left-hand side of \eqref{eq:238} coincides with the left-hand side of \eqref{eq:SchockExact}. To address the right-hand sides of the equalities, note that the asymmetric swap operator $W_{s_1,1} W_{s_2,1} \dots W_{s_m,1}$ orders the particles in $\{ -L, \dots, L\}$ according to their color, with larger colors to the left and smaller colors to the right. Thus, the position $0$ will be filled by a color $\le x$ if and only if there are at least $L+1$ particles with colors $\le x$ inside $\{ -L, \dots, L\}$ before the application of these swaps. Treating colors $\le x$ as the first class particles, and colors $>x$ as holes, we arrive at the equality of the right-hand sides of \eqref{eq:238} and \eqref{eq:SchockExact}.

\end{proof}

Next, we need the known asymptotics of a standard TASEP started from the step initial condition.

\begin{theorem}
\label{th:TasepAiry}
Assume that $u \in \R$. We have
\begin{equation}
\lim_{t \to \infty} \frac{ h^{tasep} \left( 2 u (t/2)^{2/3}, t \right) - t/4 + u (t/2)^{2/3} - u^2 t^{1/3} 2^{-4/3} }{ - t^{1/3} 2^{-4/3}} = \mathcal{A}_2 (u),
\end{equation}
where in the right-hand side $\mathcal{A}_2 (u)$ stands for the \textit{Airy process}, and the convergence is in the sense of finite-dimensional distributions.

\end{theorem}
We took the statement of this theorem (up to a slight modification of the
definition of the height function) from the survey \cite{F}, see formula (12).
Statements of this type go back to Johansson \cite{J}, who was mostly
interested in a discrete time situation; this exact claim can be obtained from
formula (2.23) in \cite{BF}, see also \cite{BFS}.

The following theorem describes the distribution of the second class particle in the GUE-GUE shock on a KPZ-type scaling.
\begin{theorem}
\label{th:shock2cl}

Assume that $L = \lfloor c t^{2/3} \rfloor$, $c \in \R_{>0}$, $y \in \R$. We have
\begin{equation}
\lim_{t \to \infty} \mathrm{Prob} \left( \frac{\mathfrak{f^{tasep}} (t)}{t^{2/3}} \le y \right) = \mathrm{Prob} \left( \mathcal{A}_2 \left( \frac{-y+c}{2^{1/3}} \right) - \mathcal{A}_2 \left( \frac{-y-c}{2^{1/3}} \right) \ge 2^{4/3} yc \right),
\end{equation}
\end{theorem}

\begin{proof}

By Proposition \ref{prop:shock}, we need to find the probability
\begin{equation}
\label{eq:neededShock}
\mathrm{Prob} \left( h^{tasep} ((-y-c)t^{2/3},t) - h^{tasep} ((-y+c)t^{2/3},t) \ge c t^{2/3}+1 \right).
\end{equation}

By Theorem \ref{th:TasepAiry}, we have
\begin{equation*}
h^{tasep} \left( (-y - c) t^{2/3}, t \right) = \frac{t}{4} + \frac{(y+c) t^{2/3}}{2} + \frac{(y+c)^2 t^{1/3}}{4} - \mathcal{A}_2 \left( \frac{-y-c}{2^{1/3}} \right) \frac{t^{1/3}}{2^{4/3}} + o \left( t^{1/3} \right),
\end{equation*}
\begin{equation*}
h^{tasep} \left( (-y + c) t^{2/3}, t \right) = \frac{t}{4} + \frac{(y-c) t^{2/3}}{2} + \frac{(y-c)^2 t^{1/3}}{4} - \mathcal{A}_2 \left( \frac{-y+c}{2^{1/3}} \right) \frac{t^{1/3}}{2^{4/3}} + o \left( t^{1/3} \right),
\end{equation*}
where $o \left( t^{1/3} \right)$ means that after division by $t^{1/3}$ this term will converge to 0 in probability. Therefore, in the $t \to \infty$ limit \eqref{eq:neededShock} converges to
\begin{multline*}
\mathrm{Prob} \left( \frac{(y+c)^2 }{4} - \mathcal{A}_2 \left( \frac{-y-c}{2^{1/3}} \right) 2^{-4/3} - \frac{(y-c)^2}{4} + \mathcal{A}_2 \left( \frac{-y+c}{2^{1/3}} \right) 2^{-4/3} \ge 0 \right)  \\
= \mathrm{Prob} \left( \mathcal{A}_2 \left( \frac{-y+c}{2^{1/3}} \right) - \mathcal{A}_2 \left( \frac{-y-c}{2^{1/3}} \right) \ge 2^{4/3} yc \right).
\end{multline*}

\end{proof}

\begin{remark}
In the setting of Theorem \ref{th:shock2cl}, we allowed the parameter of the initial condition $L$ to depend on time $t$.
Theorem \ref{th:shock2cl} describes the behavior of the second class particle in a particular joint limit of $t,L \to \infty$, with $L = c t^{2/3}$; we show that the second class particle lives on the scale $t^{2/3}$ and that the limiting distribution depends on two sections of the Airy process. Another situation was considered in Example \ref{ex:last}: there we had a finite fixed $L$. We believe that a variety of other joint limit behaviors of $t$ and $L$ also leads to a non-trivial behavior of the second class particle, and we hope to address these questions in a subsequent work. Note that Proposition \ref{prop:shock} provides an important first step in such an analysis.
\end{remark}

\section{Stochastic multicolored vertex model}

\subsection{Definition of the colored model}
\label{subsec:vertM}

Let us recall the stochastic colored vertex model (cf. \cite{BorWh}). We consider an inhomogeneous version of this model.

Consider the square grid $\mathbb Z^2$ and its first quadrant consisting of points $(m,n)$ with $m \ge 0, n \ge 0$. There is a natural partial order on the points: $(m_1,n_1) \prec (m_2,n_2)$ iff $m_1 \le m_2$ and $n_1 \le n_2$. A \textit{Ferrers diagram} $\la$ is a finite collection of points such that $(m,n) \in \la$ and $(m_1,n_1) \prec (m,n)$ implies $(m_1,n_1) \in \la$. A \textit{skew} Ferrers diagram is a set difference $\mu / \la$ of two Ferrers diagrams $\la \subset \mu$.

We supply the edges of the square grid by orientation: All vertical edges are oriented upward, and all horizontal arrows are oriented to the right. Moreover, we supply the boundary points $(m,0)$, $(0,m)$, $m \in \mathbb Z$, with arrows entering from the outside of the quadrant; thus, each vertex in the quadrant has two incoming and two outgoing arrows.

Let $S$ be a skew Ferrers diagram, let $(m_1,n_1)$ be its top left corner, and let $(m_2,n_2)$ be its bottom right corner. Then there are $M:=m_2-m_1$ vertical arrows and $N:=n_2-n_1$ horizontal arrows which enter $S$, and the same amount of vertical and horizontal arrows exit $S$. We refer to these collections of arrows as the \textit{input} and the \textit{output} of $S$, respectively. We will enumerate positions of arrows in the input and in the output by numbers from 1 to $M+N$ by following the boundary of $S$ in the counterclockwise direction (see Figure \ref{fig:14}, left panel, for an example).

\begin{figure}
\begin{center}
\noindent{\scalebox{0.7}{\includegraphics[height=7cm]{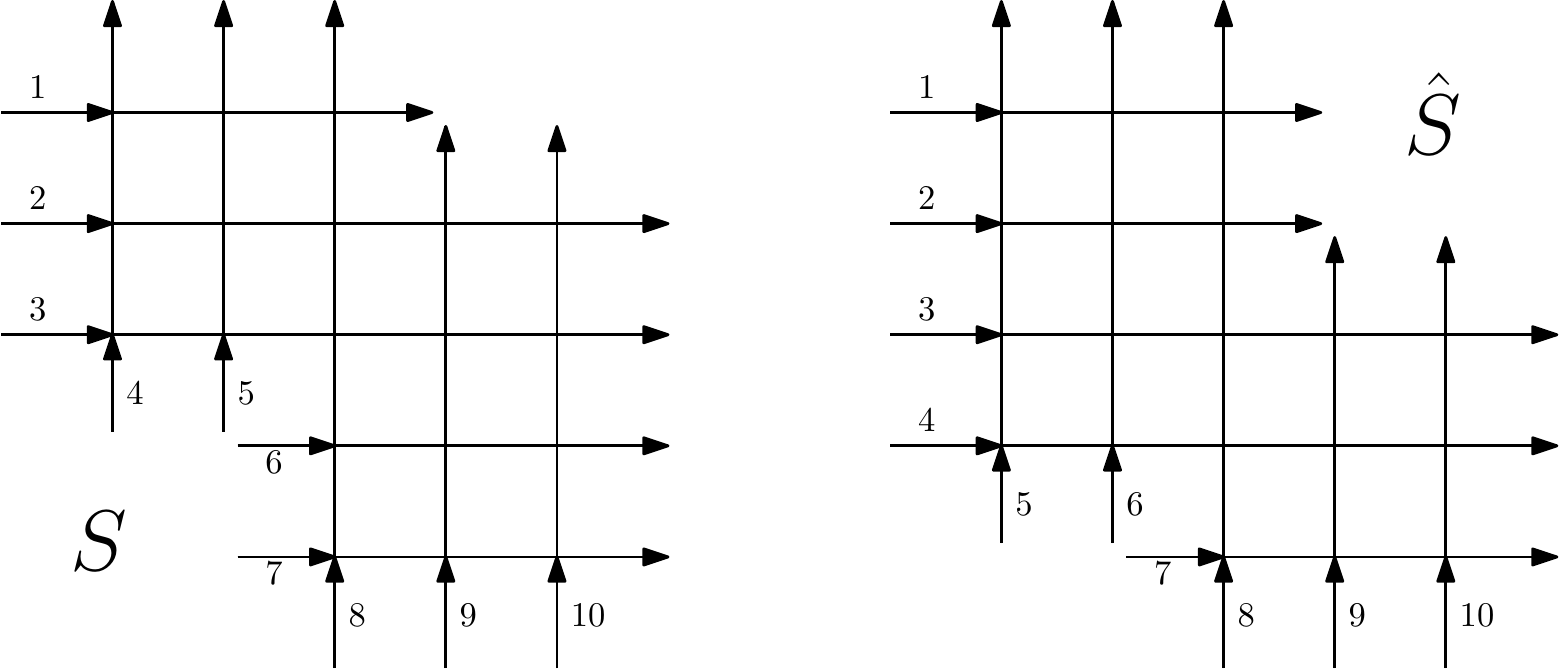}}}
\caption{ An example of a skew Ferrers diagram $S$ and its rotation by 180 degrees $\hat S$; we have $M=N=5$. The packed multi-color boundary condition is also shown.
 \label{fig:14}}
\end{center}
\end{figure}

We will enumerate colors by integers from 1 to $M+N$ as well. Let us define the colored stochastic vertex model in a skew Ferrers diagram $S$. It will depend on parameters $q$ and $\{ x_{i,j} \}_{i\ge 0, j\ge 0}$. First we assign a color to each incoming edge. Here we will mostly be interested in a specific choice of the boundary condition: We assign color $k$ to the incoming arrow at position $k$. We refer to it as the \textit{packed} initial condition.

Then, at each vertex $(i,j)$ where the colors of both incoming arrows are defined, we choose in a random way the colors of outgoing arrows; the possible variants and their probabilities are shown in Figure \ref{fig:15}, with a parameter $x = x_{i,j}$. Note that these probabilities depend on $x_{i,j}$ (which may vary between vertices), and $q$ (which is the same for all vertices). Moving in the up-right direction from the input of $S$ and iterating this procedure, we assign (randomly) colors to all arrows inside $S$ and the arrows from the output. Let us define a (random) permutation $\pi(S)$ which maps color $i$ to its output position $k$ (recall that both $i$ and $k$ range from 1 to $M+N$).

\begin{figure}
\begin{center}
\noindent{\scalebox{0.5}{\includegraphics[height=7cm]{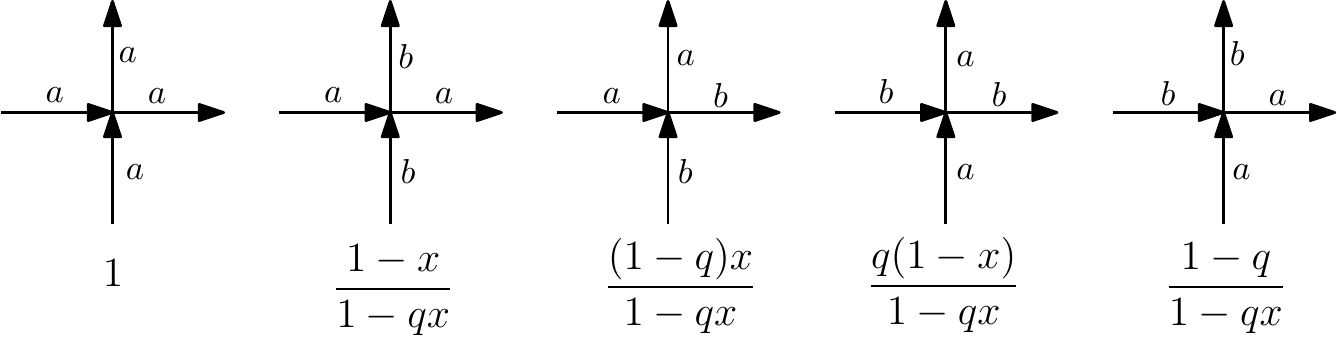}}}
\caption{ Vertices and their probabilities for colors $a<b$.
 \label{fig:15}}
\end{center}
\end{figure}

Let $\hat S$ be the rotation of a skew Ferrers diagram $S$ by 180 degrees (see Figure \ref{fig:14}, right panel). Note that this is a skew Ferrers diagram as well.

The following statement readily follows from the pictorial interpretation of Theorem \ref{th1}.
\begin{corollary}
\label{th:symmetryV}
We have
$$
\pi(S) \,{\buildrel d \over =}\, \pi^{-1} ( \hat S),
$$
where by $\,{\buildrel d \over =}\,$ we denote the equality in distribution, and the right-hand side involves the inversion in the symmetric group of $M+N$ elements.
\end{corollary}

Corollary \ref{th:symmetryV} allows to relate the distribution of the height function of the colored stochastic vertex model with those originating from the algebraic setting of symmetric functions. Below we recall necessary definitions and formulate the result.

\subsection{Hall-Littlewood processes}
\label{sec:HLdef}


A \textit{Young diagram} (or a partition) $\la$ is a finite sequence of positive integers $\la_1 \ge \la_2 \ge \cdots$. \footnote{The notion of Young diagram is equivalent to that of Ferrers diagram defined in the previous section, but we prefer to use Ferrers diagrams when speaking about collections of points in $\Z_{\ge 0}^2$, and Young diagrams when speaking about finite sequences of integers.} The \textit{length} of a partition $\la$ is the number of positive integers $\la_i$ that constitute it; it is denoted by $\la'_1$. For two Young diagrams $\la, \mu$ we write $\la \subset \mu$ if the inequalities $\la_1 \le \mu_1$, $\la_2 \le \mu_2, \dots $ hold. Similarly, we write $\lambda \prec \mu$ and say that $\la$ and $\mu$ interlace if the string of inequalities $\mu_1 \ge \la_1 \ge \mu_2 \ge \la_2 \geq \mu_3 \geq \cdots$ holds.

Let $P_{\la}$, $Q_{\la}$ be the Hall--Littlewood symmetric functions that depend on a real parameter $t$, $0 \le t < 1$ (see \cite[Chapter 3]{M}). Let $\{ a_i \}_{i=1}^{\infty}$, $\{ b_j \}_{j=1}^{\infty}$, be two collections of nonnegative reals such that $a_i b_j <1$, for all $i$ and $j$. For positive integers $M,N$, let us consider a finite sequence $s:= ( s(1), \dots, s(M+N))$, such that $s(i) \in \{-1,+1\}$, $\sum_{i=1}^{M+N} s(i) = M-N$, and $s(1)=+1$, $s(M+N)=-1$. We will also abbreviate $+1$ and $-1$ by $+$ and $-$. Let $p(i)$ be the number of pluses in the substring $(s(1), s(2), \dots, s(i))$, and let $m(i)$ be the number of minuses in the same substring.

Consider a collection of partitions $\la^{(1)} * \la^{(2)} * \dots * \la^{(M+N-1)}$, where $*$ stands for either $\subset$ or $\supset$ in the following way: if $s(i)=+$, then we have $\la^{(i-1)} \subset \la^{(i)}$; if $s(i)=-$, then $\la^{(i-1)} \supset \la^{(i)}$, and it is assumed that $\la^{0} = \la^{M+N} = \varnothing$. Set
\begin{equation}
\label{eq:HL-proc-def-1}
W^{(s,i)}_{ M,N} :=
\begin{cases}
P_{\la^{(i)} / \la^{(i-1)}} (a_{p(i)}), \qquad & \mbox{if $s(i)=+$}, \\
Q_{\la^{(i-1)} / \la^{(i)}} (b_{N-m(i)+1}), \qquad & \mbox{if $s(i)=-$}.
\end{cases}
\end{equation}

Define a \textit{Hall-Littlewood process} as a probability measure on collections of partitions $\la^{(1)}* \la^{(2)} \dots * \la^{(M+N-1)}$ by the formula
\begin{equation}
\label{eq:HL-proc-def-2}
\mathrm{Prob}_{M,N}^s ( \la^{(1)} * \la^{(2)} * \dots * \la^{(M+N-1)} ) := \frac{\displaystyle \prod_{i=1}^{M+N} W^{(s,i)}_{M,N}}{\Pi^{s}(a_1,\dots,a_M; b_1,\dots,b_N) },
\end{equation}
where $\Pi^{s}(a_1,\dots,a_M; b_1,\dots,b_N)$ is an explicit normalization constant (see \cite[Chapter 2]{BC}).

\subsection{Distribution in a color-blind model with a special form of inhomogeneities}
\label{sec:BBW}

Let us consider the following special case of the construction from Section \ref{subsec:vertM}. Let us start with a Ferrers diagram $S$ (rather than a skew Ferrers diagram).
Consider the following input: all horizontal arrows have color $1$, and all vertical arrows have color $+\infty$, so we are in a color-blind case. For $m,n \in \Z_{\ge 0}$, let $h(m,n)$ be the number of paths of color 1 which go through or below the point $(m,n)$ (see Figure \ref{fig:16} for an example). We also chose inhomogeneity parameters in a special way: For two collections $\{ a_i \}, \{ b_j \}$ as in Section \ref{sec:HLdef} we set $x_{i,j} = a_i b_j$.
The joint distribution of the output in this model can be described in terms of a Hall-Littlewood process. Let us recall the statement of this result.

\begin{figure}
\begin{center}
\noindent{\scalebox{0.62}{\includegraphics[height=7cm]{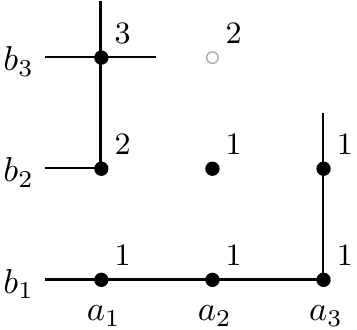}}}
\caption{ The depicted Ferrers diagram (consisting of black dots) corresponds to $s=(+,-,+,+, -, -)$. A possible configuration of paths of color 1 and corresponding values of $h(m,n)$ are shown.
 \label{fig:16}}
\end{center}
\end{figure}

The top-right boundary of $S$ can be encoded by a string $s$ of pluses and minuses via the following rule: we start at the top-left corner, and move step by step along the top-right boundary. If we are making a step to the right, we write a plus; if the step is downward, we write a minus (again, see Figure \ref{fig:16} for an example). Conversely, given a string $s$ as in Section \ref{sec:HLdef}, one can define the Ferrers diagram corresponding to it via defining its top-right boundary as a collection of points $((x_1 (s), y_1(s)), \dots, (x_{M+N-1} (s),y_{M+N-1} (s)))$ in $\mathbb Z_{\ge 0}^2$ such that:
\begin{enumerate}[{\bf 1.}]
\item $(x_1 (s) ,y_1 (s)) = (0,N-1)$, $(x_{M+N-1} (s), y_{M+N-1} (s) ) = (M-1,0)$.
\item For $2 \le i \le M+N-1$, if $s(i)=+$, then $x_{i} (s) = x_{i-1} (s) +1$, $y_{i} (s) = y_{i-1} (s)$. If $s(i)=-$, then $x_{i} (s)=x_{i-1} (s)$, $y_{i} (s)=y_{i-1} (s)-1$.
\end{enumerate}

It is clear that $s$ uniquely determines such a collection.

Let $\la (1,s) * \la (2,s) * \dots * \la (M+N-1, s)$ be distributed according to the Hall-Littlewood process defined by \eqref{eq:HL-proc-def-2} (we use the same parameters $\{ a_i \}, \{ b_j \}$, and the string $s$).

\begin{theorem}[\cite{BBW}, Theorem 4.3]
\label{th:BBW}
For any choice of parameters, the random vector $\{ h(x_i (s)+1, y_i(s)) \}_{i=1}^{M+N-1}$ has the same distribution as $\{ y_i(s) - \la'_1 (i,s) \}_{i=1}^{M+N-1}$.
\end{theorem}
See also \cite{BM}, \cite{BP} for other proofs of this result and its extensions.

\subsection{Height function: One-point distribution in the colored model}

Let $S$ be a Ferrers diagram. Consider a multi-colored vertex model in the domain $\hat S$ (it may be a \textit{skew} Ferrers diagram) via the construction from Section \ref{subsec:vertM} with the packed initial condition involving $M+N$ colors.
As inhomogeneity parameters, we first consider collections of reals $\{ a_i \}, \{ b_j \}$ as in Section \ref{sec:HLdef}, and then define the inhomogeneities $x_{i,j} := a_{M-i} b_{N-j}$ (see Figure \ref{fig:17}). Compared to Section \ref{sec:BBW} and Figure \ref{fig:16}, we reverse/rotate the order of inhomogeneities here.

\begin{figure}
\begin{center}
\noindent{\scalebox{0.7}{\includegraphics[height=7cm]{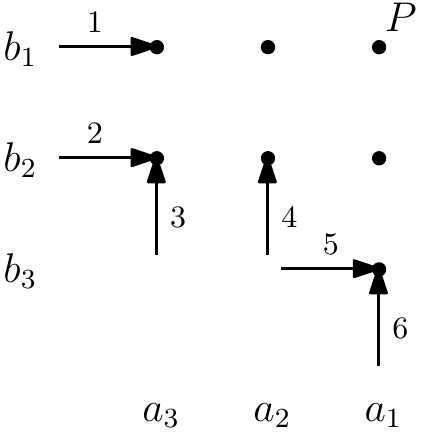}}}
\caption{ The skew Ferrers diagram $\hat S$ obtained by rotation of $S$ from Figure \ref{fig:16}. The packed colored initial condition and the reverse order of parameters are shown.
 \label{fig:17}}
\end{center}
\end{figure}

Let $P$ be the top right corner of $\hat S$ (which is unique since $S$ has the unique bottom left corner). For each $i \in \{1, \dots, M+N-1 \}$ define the colored height function $\mathcal{H}_i (P)$ to be the number of paths of colors $\le (M+N-i)$ which go through or below $P$.

\begin{theorem}
\label{th:HFmultiCvert}
For any choice of parameters, the random vector $\{\mathcal{H}_i (P) \}_{i=1}^{M+N-1}$ has the same distribution as $ \{ h(x_i (s)+1, y_i(s)) \}_{i=1}^{M+N-1}$ defined in Section \ref{sec:BBW}. Therefore, the random vector $\{ \mathcal{H}_i (P) \}_{i=1}^{M+N-1}$ has the same distribution as $\{ y_i(s) - \la'_1 (i,s) \}_{i=1}^{M+N-1}$, where $\{ \la'_1 (i,s) \}$ are lengths of partitions distributed according to the Hall-Littlewood process defined as in Theorem \ref{th:BBW}.

\end{theorem}

\begin{proof}

Let us start with the first claim. For a collection of integers $(r_1, \dots, r_{M+N-1})$, consider the event $\{ \mathcal{H}_i (P) = r_i \}_{i=1}^{M+N-1}$ and let $\nu$ be a (fixed) permutation of $(M+N)$ elements which contributes to the event (we think about $\nu$ as a map from the input to the output, see Section \ref{subsec:vertM}). Then the definitions imply that the inverse permutation $\nu^{-1}$ contributes to the event $\{ h(x_i (s)+1, y_i(s)) = r_i \}_{i=1}^{M+N-1}$. Due to Corollary \ref{th:symmetryV}, we know that they contribute the same probability to these events. Summing over all $\nu$ concludes the proof of the first claim.

The second claim immediately follows from the first one and Theorem \ref{th:BBW}.

\end{proof}

\begin{remark}
In case when both $S$ and $\hat S$ are non-skew Ferrers diagrams, a very similar claim was proved in \cite{BorWh}, cf. (1.6.5) there. Yet neither of these claims imply another one.
\end{remark}

\begin{remark}

Theorem \ref{th:BBW} allows to access a large family of observables on any down-right path of the color-blind stochastic six vertex model with inhomogeneities (see \cite{BM} for formulas for general Hall-Littlewood processes). Here we study the height function in a more complicated object: the colored stochastic vertex model. In particular, for the colored model the height function is a random vector rather than a random function. Theorem \ref{th:HFmultiCvert} gives an exact distribution of this value at any point of space. This leads to  explicit formulas for averages of various observables of this vector. The tools available for asymptotic analysis become much more numerous in the (still very nontrivial) case $q=0$, when Hall-Littlewood functions are degenerated into Schur functions. The Schur proccesses are known to enjoy a determinantal structure, which should allow for delicate asymptotic analysis of $\{\mathcal{H}_i (P) \}_{i=1}^{M+N-1}$.


\end{remark}

\begin{remark}

Similarly to Section \ref{sec:finPert}, one can apply finite perturbations to the packed initial condition; this will give a description of the output as a certain functional of the Hall-Littlewood (or Schur, for $q=0$) process.

\end{remark}

\end{document}